\documentclass[a4paper,12pt]{amsart}
\usepackage{amsmath}
\usepackage{amsthm}
\usepackage{amssymb}
\usepackage{amscd}
\usepackage[all]{xy}

\newtheorem{thm}{Theorem}[section]
\newtheorem{lemma}[thm]{Lemma}

\newcommand{\proj}{\mathop{\rm Proj}\nolimits}
\newcommand{\im}{\mathop{\rm Im}\nolimits}

\newcommand{\calhom}{\mathop{{\mathcal Hom}}\nolimits}
\DeclareMathOperator{\Hom}{Hom}
\DeclareMathOperator{\Ext}{Ext}
\DeclareMathOperator{\calext}{{\mathcal Ext}}
\DeclareMathOperator{\RHom}{RHom}
\DeclareMathOperator{\lotimes}{\otimes^{L}}
\newcommand{\caltor}{\mathop{{\mathcal T\!or}}\nolimits}
\DeclareMathOperator{\Tor}{Tor}
\DeclareMathOperator{\Ker}{Ker}
\DeclareMathOperator{\Coker}{Coker}
\DeclareMathOperator{\End}{End}
\DeclareMathOperator{\length}{length}
\DeclareMathOperator{\rk}{rank}
\DeclareMathOperator{\module}{mod}

\DeclareMathOperator{\Gr}{Gr}

\makeatletter

\@addtoreset{equation}{section}

\@addtoreset{table}{section}
\makeatother

\title[
Nef vector bundles on a quadric threefold
]{
Nef vector bundles on a 
quadric threefold
with first Chern class 
two
}

\thanks{
This work was partially supported by 
JSPS KAKENHI (C) Grant Number 21K03158.
}

\author{Masahiro Ohno
}

\address{Graduate School of Informatics and Engineering,
The University of Electro-Communications,
Chofu-shi,
Tokyo, 182-8585 Japan
}

\email{masahiro-ohno@uec.ac.jp}

\subjclass[2020]{Primary 
14J60;
Secondary 
14J45,
14F08,
}

\keywords{nef vector bundles,
Fano bundles, full strong exceptional collections}

\begin{document}
\begin{abstract}
We classify nef vector bundles on a smooth hyperquadric
of dimension 
three
with first Chern class 
two
over an algebraically closed field of characteristic zero.
In particular, we see that they are globally generated.
\end{abstract}

\maketitle


\section{Introduction}
In 
\cite[\S~2 Theorem~2]{pswnef},
Peternell-Szurek-Wi\'{s}niewski
classified
nef vector bundles on a smooth hyperquadric $\mathbb{Q}^n$ of dimension $n\geq 3$ 
with first Chern class $\leq 1$
over an algebraically closed field $K$ of characteristic zero.
In \cite[Theorem~9.3]{MR4453350}, we 
provided
a different proof of this classification,
which was based on an analysis with a full strong exceptional collection
of vector bundles on $\mathbb{Q}^n$.

In this paper, we 
classify nef vector bundles on
a smooth quadric threefold 
$\mathbb{Q}^3$   
with
first Chern class two.
(In the subsequent paper \cite{HyperquadricOfDim4c1=2}, we classify 
those on a smooth hyperquadric $\mathbb{Q}^n$ of dimension $n\geq 4$.)
The precise statement is as follows.

\begin{thm}\label{Chern2}
Let $\mathcal{E}$ be a nef vector bundle of rank $r$
on a smooth hyperquadric $\mathbb{Q}^3$ of dimension $3$
over an algebraically closed field $K$ 
of characteristic zero,
and let $\mathcal{S}$ be the spinor bundle on $\mathbb{Q}^3$.
Suppose that 
$\det\mathcal{E}\cong \mathcal{O}(2)$.
Then 
$\mathcal{E}$ is isomorphic to one of the following vector bundles
or fits in one of the following exact sequences:
\begin{enumerate}
\item[(1)]
$\mathcal{O}(2)\oplus \mathcal{O}^{\oplus r-1}$;
\item[(2)] $\mathcal{O}(1)^{\oplus 2}\oplus\mathcal{O}^{\oplus r-2}$;
\item[(3)] $\mathcal{O}(1)\oplus 
\mathcal{S}\oplus
\mathcal{O}^{\oplus r-3}$;
\item[(4)] $0\to \mathcal{O}(-1)\to \mathcal{O}(1)\oplus 
\mathcal{O}^{\oplus r}
\to \mathcal{E}\to 0$;
\item[(5)] $0\to \mathcal{O}^{\oplus a}\to\mathcal{S}^{\oplus 2}\oplus\mathcal{O}^{\oplus r-4+a}
\to \mathcal{E}\to 0$, where $a=0$ or $1$,
and the composite of the injection 
$\mathcal{O}^{\oplus a}\to\mathcal{S}^{\oplus 2}\oplus\mathcal{O}^{\oplus r-4+a}$
and the projection $\mathcal{S}^{\oplus 2}\oplus\mathcal{O}^{\oplus r-4+a}
\to \mathcal{O}^{\oplus r-4+a}$ is zero;
\item[(6)] $0\to \mathcal{S}(-1)\oplus  \mathcal{O}(-1)
\to \mathcal{O}^{\oplus r+3}
\to \mathcal{E}\to 0$;
\item[(7)] $0\to \mathcal{O}(-1)^{\oplus 2}\to 
\mathcal{O}^{\oplus r+2}
\to \mathcal{E}\to 0$;
\item[(8)] $0\to \mathcal{O}(-2)\to \mathcal{O}^{\oplus r+1}
\to \mathcal{E}\to 0$;
\item[(9)] $0\to \mathcal{O}(-2)\to 
\mathcal{O}(-1)^{\oplus 4}
\to 
\mathcal{O}^{\oplus r+3}
\to \mathcal{E}\to 0$.
\end{enumerate}
\end{thm}
Note that this list is effective: in each case exists an example.
For example, if we denote by $\mathcal{N}$ 
a null correlation bundle on $\mathbb{P}^3$,
then $\pi_p^*(\mathcal{N})$ belongs to Case $(9)$ of Theorem~\ref{Chern2},
where $\pi_p:\mathbb{Q}^3\to \mathbb{P}^3$ is the projection from a point 
$p\in \mathbb{P}^3\setminus\mathbb{Q}^3$.
(Similarly, $\pi_p^*(\Omega_{\mathbb{P}^3}(2))$ belongs to Case $(9)$ of Theorem~\ref{Chern2}.)
Under the stronger assumption that $\mathcal{E}$ is globally generated,
Ballico-Huh-Malaspina 
provided
a classification of $\mathcal{E}$
on $\mathbb{Q}^3$ with $c_1=2$ 
in \cite{MR3509232} and \cite{MR3120621}.

Note also that the projectivization $\mathbb{P}(\mathcal{E})$
of the bundle $\mathcal{E}$ in Theorem~\ref{Chern2}
is a Fano manifold of dimension $r+2$, i.e., the bundle $\mathcal{E}$
in Theorem~\ref{Chern2} is a Fano bundle  on $\mathbb{Q}^3$ of rank $r$.
As a related result, Langer classified smooth Fano $4$-folds with adjunction theoretic scroll
structure over $\mathbb{Q}^3$ in \cite[Theorem 7.2]{MR1633159}.

Our basic strategy and framework for describing  
$\mathcal{E}$
in Theorem~\ref{Chern2} 
is to give a minimal locally free resolution 
of $\mathcal{E}$
in terms of some twists of the full strong exceptional collection 
\[(\mathcal{O},\mathcal{S},\mathcal{O}(1),\mathcal{O}(2))\] 
of vector bundles (see \cite{MR4453350} for more details).

The content of this paper is as follows.
In Section~\ref{Preliminaries}, we briefly recall 
Bondal's theorem \cite[Theorem 6.2]{MR992977}
and its related notions
and results
required
in the proof of Theorem~\ref{Chern2}.
In particular, we recall some finite-dimensional algebra $A$,
and fix some symbols, 
e.g., $G$, $P_i$, and $S_i$, 
related to $A$ and to finitely generated right $A$-modules. 
We also recall the classification 
\cite[Theorem 1.1]{QuadricSurfacec1=2}
of nef vector bundles
on a smooth quadric surface $\mathbb{Q}^2$ with Chern class $(2,2)$
in Theorem~\ref{Chern(2,2)}.
In Section~\ref{PropertyOfSpinors}, we recall some basic properties 
of the spinor bundle $\mathcal{S}$ on $\mathbb{Q}^3$.
In Section~\ref{HRR}, we state Hirzebruch-Riemann-Roch formulas
for vector bundles $\mathcal{E}$ on $\mathbb{Q}^3$ with $c_1=2$
and  for $\mathcal{S}^\vee\otimes \mathcal{E}$.
In Section~\ref{Key Lemmas},
we show some key lemmas 
required 
later in the proof of Theorem~\ref{Chern2}.
In Section~\ref{lowerBound}, we provide a lower bound 
for the third Chern class of a nef vector bundle $\mathcal{E}$,
if $h^0(\mathcal{E}(-D))\neq 0$ for some effective divisor $D$.
In Section~\ref{Set-up for the case $n=3$}, 
we provide
the set up for the proof of Theorem~\ref{Chern2}.
The proof of Theorem~\ref{Chern2} is carried out 
in Sections~\ref{Case(1)OfTheoremChern(2,2)} to \ref{Case(13)or(14)OfTheorem(2,2)},
according to which case of Theorem~\ref{Chern(2,2)} $\mathcal{E}|_{\mathbb{Q}^2}$ belongs to.

\subsection{Notation and conventions}\label{convention}
Throughout this paper,
we work over an algebraically closed field $K$ of characteristic zero.
Basically we follow the standard notation and terminology in algebraic
geometry. 
We denote 
by $\mathbb{Q}^3$ a smooth quadric threefold over $K$,
by $\mathbb{Q}^2$ a smooth quadric surface over $K$,
and by 
\begin{itemize}
\item $\mathcal{S}$ the spinor bundle on $\mathbb{Q}^3$.
\end{itemize}
Note that we follow Kapranov's convention~\cite[p.\ 499]{MR0939472};
our spinor bundle $\mathcal{S}$ is globally generated,
and it is the dual of that of Ottaviani's~\cite{ot}.
For a coherent sheaf $\mathcal{F}$,
we denote by $c_i(\mathcal{F})$ the $i$-th Chern class of $\mathcal{F}$
and by $\mathcal{F}^{\vee}$ the dual of $\mathcal{F}$.
In particular, 
\begin{itemize}
\item $c_i$ stands for $c_i(\mathcal{E})$
of the nef vector bundle $\mathcal{E}$ we are dealing with.
\end{itemize}
For a vector bundle $\mathcal{E}$,
$\mathbb{P}(\mathcal{E})$ denotes $\proj S(\mathcal{E})$,
where $S(\mathcal{E})$ denotes the symmetric algebra of $\mathcal{E}$.
The tautological line bundle 
\begin{itemize}
\item $\mathcal{O}_{\mathbb{P}(\mathcal{E})}(1)$
is also denoted by $H(\mathcal{E})$.
\end{itemize}
Let $A^*\mathbb{Q}^3$ be the Chow ring of $\mathbb{Q}^3$.
We denote
\begin{itemize}
\item by $H$ a hyperplane section of $\mathbb{Q}^3$ and 
by $h$ its class in $A^1\mathbb{Q}^3$: $A^1\mathbb{Q}^3=\mathbb{Z}h$;
\item by $L$ a line in $\mathbb{Q}^3$ and 
by $l$ its class in $A^2\mathbb{Q}^3$:
$A^2\mathbb{Q}^3=\mathbb{Z}l$.
\end{itemize}
Note that $h^2=2l$.
Via the map $\deg:A^3\mathbb{Q}^3\cong \mathbb{Z}$, we identify 
elements $A^3\mathbb{Q}^3$ with 
its corresponding integer; thus we have $h^3=2$ and $hl=1$.
For any closed subscheme $Z$
in $\mathbb{Q}^3$, 
$\mathcal{I}_Z$
denotes
the ideal sheaf of $Z$ in $\mathbb{Q}^3$;
for a point $p\in \mathbb{Q}^3$, $\mathcal{I}_p$ denotes
the ideal sheaf of 
$p\in \mathbb{Q}^3$
and $k(p)$ denotes the residue field of $p\in \mathbb{Q}^3$.
Finally we refer to \cite{MR2095472} for the definition
and basic properties of nef vector bundles.

\section{Preliminaries}\label{Preliminaries}
Throughout this paper, 
$G_0$, $G_1$, $G_2$, $G_3$ denote respectively 
$\mathcal{O}$, $\mathcal{S}$, $\mathcal{O}(1)$, $\mathcal{O}(2)$
on $\mathbb{Q}^3$.
An important and  well known fact~\cite[Theorem~4.10]{MR0939472} of the 
collection 
$(G_0,G_1,G_2,G_3)$
is that it is a full strong exceptional collection in 
$D^b(\mathbb{Q}^3)$,
where $D^b(\mathbb{Q}^3)$ denotes the bounded derived category
of (the abelian category of) coherent sheaves on $\mathbb{Q}^3$.
Here we use the term ``collection'' to mean ``family'', not ``set''. 
Thus an exceptional collection is also called an exceptional sequence. 
We refer to \cite{MR2244106} for the definition of a full strong exceptional sequence.

Denote by $G$ the direct sum $\bigoplus_{i=0}^3G_i$ of $G_0$, $G_1$, $G_2$, and $G_3$,
and by $A$ the endomorphism ring $\End(G)$ of $G$.
The ring  $A$ is a finite-dimensional $K$-algebra,
and $G$ is a left $A$-module.
Note that $\Ext^q(G, \mathcal{F})$ is a finitely generated 
right $A$-module for a coherent sheaf $\mathcal{F}$ 
on $\mathbb{Q}^3$.
We denote by $\module A$ the category of finitely generated right $A$-modules
and by $D^b(\module A)$ the bounded derived category
of $\module A$.
Let $p_i:G\to G_i$ be the projection,
and $\iota_i:G_i\hookrightarrow G$ the inclusion.
Set $e_i=\iota_i\circ p_i$. Then $e_i\in A$.
Set 
\[P_i=e_iA.\]
Then $A\cong \oplus_i P_i$ as right $A$-modules,
and $P_i$'s are projective right $A$-modules.
We see that $P_i\otimes_A G\cong G_i$.
Any finitely generated right $A$-module $V$
has an ascending filtration 
\[0=V^{\leq -1}\subset V^{\leq 0}\subset V^{\leq 1}\subset 
V^{\leq 2}
\subset V^{\leq 3}=V\]
by right $A$-submodules,
where $V^{\leq i}$ is defined to be $\bigoplus_{j\leq i}Ve_j$.
Set $\Gr^iV=V^{\leq i}/V^{\leq i-1}$
and 
\[S_i=\Gr^iP_i.\]
Then 
$\Gr^iS_i\cong K$ as $K$-vector spaces,
$\Gr^jS_i=0$ for any $j\neq i$,
and $S_i$ is a simple right $A$-module.
If we set $m_i=\dim_K\Gr^iV$, then 
$\Gr^iV\cong S_i^{\oplus m_i}$
as right $A$-modules.

It follows from Bondal's theorem \cite[Theorem 6.2]{MR992977}
that 
\[\RHom(G,\bullet):D^b(\mathbb{Q}^3)\to D^b (\module A)\]
is an exact equivalence,
and its quasi-inverse is 
\[\bullet\lotimes_AG:D^b(\module A)\to D^b(\mathbb{Q}^3).\]
For a coherent sheaf $\mathcal{F}$ on $\mathbb{Q}^3$, this fact can be rephrased 
in terms of a spectral sequence\cite[Theorem 1]{MR3275418}:
\begin{equation}\label{BondalSpectral}
E_2^{p,q}=\caltor_{-p}^A(\Ext^q(G,\mathcal{F}),G)
\Rightarrow
E^{p+q}=
\begin{cases}
\mathcal{F}& \textrm{if}\quad  p+q= 0\\
0& \textrm{if}\quad  p+q\neq 0,
\end{cases}
\end{equation}
which is called  the Bondal spectral sequence.
Note that $E_2^{p,q}$ is the $p$-th cohomology sheaf 
$\mathcal{H}^p(\Ext^q(G,\mathcal{F})\lotimes_AG)$ of the complex 
$\Ext^q(G,\mathcal{F})\lotimes_AG$.
When we compute the spectral sequence, we consider the descending filtration
on the right $A$-module $\Ext^q(G,\mathcal{F})$, and apply the following 

\begin{lemma}\label{S2Arilemma}
We have 
\begin{align}
S_3\lotimes_AG&\cong \mathcal{O}(-1)[3];\label{S3O(-1)}\\ 
S_2\lotimes_AG&\cong T_{\mathbb{P}^4}(-2)|_{\mathbb{Q}^3}[2];\label{S2F}\\
S_1\lotimes_AG&\cong \mathcal{S}^{\vee}[1]\cong \mathcal{S}(-1)[1];
\label{S1Sdual}\\
S_0\lotimes_AG&\cong \mathcal{O},\label{S0O}
\end{align}
where $T_{\mathbb{P}^4}$ denotes the tangent bundle of $\mathbb{P}^4$.
\end{lemma}
\begin{proof}
Since $\RHom(G,\mathcal{O}(-1)[3])\cong S_3$,
we obtain \eqref{S3O(-1)}.
Note that we have an isomorphism
$\RHom(G,\mathcal{S}^{\vee}[1])\cong S_1$ by \cite[Lemma 8.2 (1)]{MR4453350}.
Hence 
we have \eqref{S1Sdual}.
It is easy to see that the last isomorphism \eqref{S0O} holds.
To see \eqref{S2F}, first
note that  we have the following exact sequence:
\[
0\to \mathcal{O}(-2)\to 
\mathcal{O}(-1)\otimes H^0(\mathcal{O}(1))^{\vee}
\to T_{\mathbb{P}^4}(-2)|_{\mathbb{Q}^3}\to 0. 
\]
Serre duality shows that  
\[H^3(\mathcal{O}(-4))\to 
H^3(\mathcal{O}(-3))\otimes H^0(\mathcal{O}(1))^{\vee}\]
is dual of the canonical isomorphism
\[
H^0(\mathcal{O})\otimes H^0(\mathcal{O}(1))
\to H^0(\mathcal{O}(1)).
\]
Hence $H^q(T_{\mathbb{P}^4}(-4)|_{\mathbb{Q}^3})=0$ for all $q$.
Moreover $h^q(\mathcal{S}^{\vee}(-i))=0$ for $i=0,1,2$ and all $q$.
Therefore we conclude that $\RHom(G,T_{\mathbb{P}^4}(-2)|_{\mathbb{Q}^3}))$ is isomorphic to $S_2[-2]$.
\end{proof}

Our proof of Theorem~\ref{Chern2} relies on 
the following
theorem~\cite[Theorem 1.1]{QuadricSurfacec1=2}:
\begin{thm}\label{Chern(2,2)}
Let $\mathcal{E}$ be a nef vector bundle of rank $r$
on a smooth quadric surface $\mathbb{Q}^2$
over an algebraically closed field $K$ 
of characteristic zero.
Suppose that 
$\det\mathcal{E}\cong \mathcal{O}(2,2)$.
Then 
$\mathcal{E}$ is isomorphic to one of the following vector bundles
or fits in one of the following exact sequences:
\begin{enumerate}
\item[(1)]
$\mathcal{O}(2,2)\oplus \mathcal{O}^{\oplus r-1}$;
\item[(2)] 
$\mathcal{O}(2,1)\oplus\mathcal{O}(0,1)\oplus  \mathcal{O}^{\oplus r-2}$;
\\
$\mathcal{O}(1,2)\oplus
\mathcal{O}(1,0)\oplus  \mathcal{O}^{\oplus r-2}$;
\\
$($We do not exhibit the cases
obtained by replacing $(a,b)$ with $(b,a)$ in the following:$)$
\item[$(3)$] $\mathcal{O}(1,1)^{\oplus 2}
\oplus  \mathcal{O}^{\oplus r-2}$;
\item[$(4)$] 
$0\to \mathcal{O}
\xrightarrow{\iota}
\mathcal{O}(1,1)\oplus
\mathcal{O}(1,0)\oplus  
\mathcal{O}(0,1)\oplus  
\mathcal{O}^{\oplus r-2}\to \mathcal{E}\to 0$;\\
$($More precisely, either the composite of $\iota$ and the projection 
\[\mathcal{O}(1,1)\oplus
\mathcal{O}(1,0)\oplus  
\mathcal{O}(0,1)\oplus  
\mathcal{O}^{\oplus r-2}\to \mathcal{O}\]
is zero, or 
$\mathcal{E}$ is isomorphic to $\mathcal{O}(1,1)\oplus
\mathcal{O}(1,0)\oplus  
\mathcal{O}(0,1)\oplus  
\mathcal{O}^{\oplus r-3}$.$)$
\item[$(5)$] $0\to \mathcal{O}(-1,-1)\to \mathcal{O}(1,1)\oplus
\mathcal{O}^{\oplus r}\to \mathcal{E}\to 0$;
\item[$(6)$]  $0\to \mathcal{O}^{\oplus 2}\to \mathcal{O}(1,0)^{\oplus 2}\oplus
\mathcal{O}(0,1)^{\oplus 2}\oplus  
\mathcal{O}^{\oplus r-2}\to \mathcal{E}\to 0$;
\\
$($More precisely, 
\begin{enumerate}
\item[$(6$-$1)$] 
$0\to \mathcal{O}\to \mathcal{O}(2,0)\oplus
\mathcal{O}(0,1)^{\oplus 2}\oplus  
\mathcal{O}^{\oplus r-2}\to \mathcal{E}\to 0$;
\\ In this case, we have
\begin{enumerate}
\item[$(6$-$1$-$1)$] $\mathcal{O}(2,0)\oplus\mathcal{O}(0,2)\oplus\mathcal{O}^{\oplus r-2}$;
\item[$(6$-$1$-$2)$] $\mathcal{O}(2,0)\oplus\mathcal{O}(0,1)^{\oplus 2}\oplus  \mathcal{O}^{\oplus r-3}$;
\end{enumerate}
\item[$(6$-$2)$] $\mathcal{O}(1,0)^{\oplus 2}\oplus  
\mathcal{O}(0,1)^{\oplus 2}\oplus  
\mathcal{O}^{\oplus r-4}$;
\item[$(6$-$3)$]
$0\to \mathcal{O}(0,-1)\to \mathcal{O}(1,0)^{\oplus 2}\oplus  
\mathcal{O}(0,1)\oplus  
\mathcal{O}^{\oplus r-2}\to \mathcal{E}\to 0.)$
\end{enumerate}
\item[$(7)$] $0\to \mathcal{O}(-1,-1)
\oplus \mathcal{O}(-1,0)
\oplus \mathcal{O}(0,-1)
\to \mathcal{O}^{\oplus r+3}\to \mathcal{E}\to 0$;
\item[$(8)$] $0\to \mathcal{O}(-1,-2)
\to \mathcal{O}(1,0)\oplus 
\mathcal{O}^{\oplus r}\to \mathcal{E}\to 0$;
\item[$(9)$] $0\to \mathcal{O}(-1,-1)^{\oplus 2}\to   
\mathcal{O}^{\oplus r+2}\to \mathcal{E}\to 0$;
\item[$(10)$] $0\to \mathcal{O}(-2,-2)\to   
\mathcal{O}^{\oplus r+1}\to \mathcal{E}\to 0$;
\item[$(11)$] $0\to \mathcal{O}(-2,-2)\to   
\mathcal{O}^{\oplus r+1}\to \mathcal{E}\to k(p)\to 0$;
\item[$(12)$] $0\to \mathcal{O}(-2,-2)\to   
\mathcal{O}^{\oplus r}\to \mathcal{E}\to \mathcal{O}\to 0$;
\item[$(13)$] $0\to \mathcal{O}(-1,-1)^{\oplus 4}\to   
\mathcal{O}^{\oplus r}\oplus\mathcal{O}(-1,0)^{\oplus 2}
\oplus \mathcal{O}(0,-1)^{\oplus 2}
\to \mathcal{E}\to 0$.
\end{enumerate}
\end{thm}

\section{Some basic properties of the spinor bundle $\mathcal{S}$
on $\mathbb{Q}^3$}\label{PropertyOfSpinors}
We recall some basic facts and properties 
of the spinor bundle $\mathcal{S}$ on $\mathbb{Q}^3$
in our notation
(see Ottaviani's result~\cite{ot} and \cite[Theorem 8.1]{MR4453350}).
First we have an exact sequence
\begin{equation}\label{SSdual}
0\to \mathcal{S}^{\vee}\to \mathcal{O}^{\oplus 4}\to \mathcal{S}\to 0
\end{equation}
by \cite[Theorem~2.8 (1)]{ot}.
The restriction $\mathcal{S}|_{\mathbb{Q}^2}$
of $\mathcal{S}$ to a smooth hyperplane section $\mathbb{Q}^2$ of $\mathbb{Q}^3$
is isomorphic to $\mathcal{O}(1,0)\oplus \mathcal{O}(0,1)$,
and $h^0(\mathcal{S})=4$.
We have $\det\mathcal{S}=\mathcal{O}(1)$,
and thus the canonical isomorphism 
\begin{equation}\label{canonicalIsom}
\mathcal{S}^{\vee}(1)\cong \mathcal{S}.
\end{equation}
The zero locus $(s)_0$ of every non-zero element $s$ of $H^0(\mathcal{S})$
is a line $l$ in $\mathbb{Q}^3$. 
Thus $c_1(\mathcal{S})\cap [\mathbb{Q}^3]=h$
and $c_2(\mathcal{S})\cap [\mathbb{Q}^3]=l$.
We have $h^q(\mathcal{S})=0$ for any $q>0$
and $h^q(\mathcal{S}(-i))=0$ for all $q$ if $i=1,2,$ or $3$.

\begin{lemma}\label{wedgeproductOfS}
The natural map
\[
H^0(\mathcal{S})\otimes H^0(\mathcal{S})\to H^0(\mathcal{O}(1))
\]
sending $s\otimes t$ to $s\wedge t$ is surjective.
\end{lemma}
\begin{proof}
Without loss of generality,
we may assume that $\mathbb{Q}^3$ is defined by 
an equation $X_{01}^2-X_{02}X_{13}+X_{03}X_{12}=0$,
where 
$[X_{01}:X_{02}:X_{03}:X_{12}:X_{13}]$ 
is the homogeneous coordinates of 
$\mathbb{P}^4$.
We may also 
regard $\mathbb{Q}^3$ as a smooth hyperplane section $H\cap \mathbb{Q}^4$
of a smooth hyperquadric $\mathbb{Q}^4$
defined by an equation $X_{01}X_{23}-X_{02}X_{13}+X_{03}X_{12}=0$,
where $X_{ij}$ ($0\leq i<j\leq 3$) are homogeneous coordinates of $\mathbb{P}^5$,
and $H$ is the hyperplane defined by $X_{01}=X_{23}$.
Note that 
$\mathbb{Q}^4$ is the image of
the Grassmannian $G(1,3)$ parametrizing lines in $\mathbb{P}^3$
by the Pl\"{u}cker embedding $\iota$.
If we represent a point in $G(1,3)$ by a matrix
$\begin{bmatrix}
x_{10}&x_{11}&x_{12}&x_{13}\\
x_{20}&x_{21}&x_{22}&x_{23}
\end{bmatrix}$,
then $\iota^*X_{ij}=
\begin{vmatrix}
x_{1i}&x_{1j}\\
x_{2i}&x_{2j}
\end{vmatrix}$.
We will identify $\mathbb{Q}^4$ with $G(1,3)$ via $\iota$. 
Let 
$H^0(\mathbb{P}^3,\mathcal{O}(1))\otimes \mathcal{O}_{G(1,3)} \to 
\mathcal{Q}$
be the universal quotient bundle on $G(1,3)$,
which sends homogeneous coordinates $x_j$ of $\mathbb{P}^3$
to global sections 
$s_j$ 
of 
$\mathcal{Q}$
represented by 
$\begin{bmatrix}
x_{1j}\\
x_{2j}
\end{bmatrix}$.
Recall that $\mathcal{S}$ is the restriction of $\mathcal{U}$
to the hyperplane section $H\cap \mathbb{Q}^4=\mathbb{Q}^3$.
By abuse of notation, we will denote by $s_j$ the restriction of $s_j$ to 
$\mathbb{Q}^3$.
Since $h^0(\mathcal{S})=4$,
$H^0(\mathcal{S})$ is spanned by $s_0,s_1,s_2,s_3$.
Moreover 
$H^0(\mathcal{O}(1))$ is 
spanned
by $X_{i,j}=s_i\wedge s_j$,
where $(i,j)=(0,1), (0,2), (0,3), (1,2)$, and $(1,3)$. 
This completes the proof.
\end{proof}

\section{Hirzebruch-Riemann-Roch formulas}\label{HRR}
Let $\mathcal{E}$ be a vector bundle of rank $r$ on $\mathbb{Q}^3$.
Since the tangent bundle $T$ of $\mathbb{Q}^3$ fits in an exact sequence
\[0\to T\to T_{\mathbb{P}^4}|_{\mathbb{Q}^3}
\to \mathcal{O}_{\mathbb{Q}^3}(2)\to 0,\]
the Chern polynomial $c_t(T)$ of $T$ is 
\[\dfrac{(1+ht)^5}{1+2ht}=1+3ht+4h^2t^2+2h^3t^3,\]
where $h$ denotes $c_1(\mathcal{O}_{\mathbb{Q}^3}(1))$.
Then the Hirzebruch-Riemann-Roch formula implies that
\[
\chi(\mathcal{E})=r+\dfrac{13}{12}c_1h^2
+\dfrac{3}{4}(c_1^2-2c_2)h+
\dfrac{1}{6}(c_1^3-3c_1c_2+3c_3),
\]
where we set $c_i=c_i(\mathcal{E})$.
To compute $\chi(\mathcal{E}(t))$, note that 
\begin{align*}
c_1(\mathcal{E}(t))&=c_1+rth;\\
c_2(\mathcal{E}(t))&
=c_2+(r-1)tc_1h+
\binom{r}{2}
t^2h^2;\\
c_3(\mathcal{E}(t))&=c_3
+(r-2)tc_2h+
\binom{r-1}{2}t^2c_1
h^2+\binom{r}{3}t^3h^3.
\end{align*}
Since $h^3=2$, we infer that 
\begin{equation}\label{generalRRonQ3}
\begin{split}
\chi(\mathcal{E}(t))=&\dfrac{r}{3}t^3+
\dfrac{1}{2}(c_1h^2+3r)t^2
+\dfrac{1}{2}\{3c_1h^2
+(c_1^2-2c_2)h+\dfrac{13}{3}r\}t\\
&+r+\dfrac{13}{12}c_1h^2
+\dfrac{3}{4}(c_1^2-2c_2)h
+\dfrac{1}{6}(c_1^3-3c_1c_2+3c_3).
\end{split}
\end{equation}
Since $c_1(\mathcal{E})=dh$ for some integer $d$,
the formula above can be written as
\begin{equation}
\begin{split}
\chi(\mathcal{E}(t))=&
\dfrac{r}{6}(2t+3)(t+2)(t+1)+dt^2+(d^2+3d)t-c_2ht
\\
&+\dfrac{d}{6}(2d^2+9d+13)
+\dfrac{1}{2}\{c_3-(d+3)c_2h\}.
\end{split}
\end{equation}
In this paper, we are dealing with the case $d=2$:
\begin{equation}\label{e(t)RR}
\chi(\mathcal{E}(t))=
\dfrac{r}{6}(2t+3)(t+2)(t+1)+2t^2+10t+13-c_2ht
+\dfrac{1}{2}\{c_3-5c_2h\}.
\end{equation}
In particular,
\begin{align}
\chi(\mathcal{E}(-1))&=
5-\dfrac{3}{2}c_2h
+\dfrac{1}{2}c_3;\label{e(-1)RR}\\
\chi(\mathcal{E}(-2))&=
1-\dfrac{1}{2}c_2h
+\dfrac{1}{2}c_3.\label{e(-2)RR}
\end{align}
Next we will compute $\chi(\mathcal{S}^{\vee}\otimes \mathcal{E}(t))$.
Recall that $c_1(\mathcal{S})=h$ and that $c_1(\mathcal{S})c_2(\mathcal{S})=1$.
Note also that 
\begin{align*}
\rk \mathcal{S}^{\vee}\otimes \mathcal{E}=&2r;\\
c_1(\mathcal{S}^{\vee}\otimes \mathcal{E})=&2c_1-rh;\\
c_2(\mathcal{S}^{\vee}\otimes \mathcal{E})=
&2c_2-(2r-1)c_1h
+c_1^2+\binom{r}{2}h^2+rc_2(\mathcal{S});\\
c_3(\mathcal{S}^{\vee}\otimes\mathcal{E})=
&2c_3-2(r-1)c_2h+(r-1)^2c_1h^2
+2(r-1)c_1c_2(\mathcal{S})\\
&+2c_1c_2-(r-1)c_1^2h-\dfrac{1}{3}r(r^2-1).
\end{align*}
The formula \eqref{generalRRonQ3} together with the formulas above 
implies the following formula:
\begin{align*}
\chi(\mathcal{S}^{\vee}\otimes \mathcal{E}(t))=&\dfrac{2}{3}rt^3+
(c_1h^2+2r)t^2
+\{2c_1h^2+(c_1^2-2c_2)h+\dfrac{4}{3}r\}t\\
&+\dfrac{7}{6}c_1h^2+c_1^2h-2c_2h
+\dfrac{1}{3}c_1^3
+c_3-c_1c_2-c_1c_2(\mathcal{S}).
\end{align*}
Since $c_1=dh$, the formula above becomes the following formula:
\begin{equation}\label{SobokuRRwithSpinor}
\begin{split}
\chi(\mathcal{S}^{\vee}\otimes \mathcal{E}(t))=&
\dfrac{2}{3}rt(t+1)(t+2)+2dt^2+2d(d+2)t\\
&+
\dfrac{2}{3}d(d+1)(d+2)-(2t+d+2)c_2h+c_3.
\end{split}
\end{equation}
For the case $d=2$, we have
\begin{equation}\label{SERR}
\chi(\mathcal{S}^{\vee}\otimes \mathcal{E}(t))=
\dfrac{2}{3}rt(t+1)(t+2)+4(t+2)^2
-2(t+2)c_2h+c_3.
\end{equation}
In particular,
\begin{equation}\label{SERR(-1)}
\chi(\mathcal{S}^{\vee}\otimes \mathcal{E}(-1))=
4-2c_2h+c_3.
\end{equation}

\section{Key Lemmas}\label{Key Lemmas}
\begin{lemma}\label{S2nukelemma}
We have the following exact sequence on $\mathbb{Q}^3$:
\begin{equation}\label{OneOfKeyExactSeq}
0\to T_{\mathbb{P}^4}(-2)|_{\mathbb{Q}^3}
\to \mathcal{S}^{\vee \oplus 4}\to \Omega_{\mathbb{P}^4}(1)|_{\mathbb{Q}^3}\to 0.
\end{equation}
\end{lemma}
\begin{proof}
We will show \eqref{OneOfKeyExactSeq} by 
showing that 
$T_{\mathbb{P}^4}(-2)|_{\mathbb{Q}^3}$ fits in the following exact sequence:
\[
0\to T_{\mathbb{P}^4}(-2)|_{\mathbb{Q}^3}\to \mathcal{O}^{\oplus 11}\to \mathcal{S}^{\oplus 4}\to \mathcal{O}(1)\to 0.
\]
As we already know that $S_2\lotimes_AG\cong T_{\mathbb{P}^4}(-2)|_{\mathbb{Q}^3}[2]$,
it is enough to show that 
$S_2\lotimes_AG$ is nothing but the following complex:
\[
0\to 
\mathcal{O}^{\oplus 11}
\to 
\mathcal{S}^{\oplus 4}
\to \mathcal{O}(1)\to 0,
\]
where $\mathcal{O}(1)$ lies in degree zero.
To show this, we claim that 
$S_2$ has the following projective resolution of right $A$-modules:
\[0\to P_0^{\oplus 11}\to 
\Hom(\mathcal{S},\mathcal{O}(1))\otimes_K P_1
\to P_2\to S_2\to 0.\] 
Indeed, it is clear that we have a surjection $P_2\to S_2$. 
Since the kernel $N$ of the surjection $P_2\to S_2$ is isomorphic to 
$\Hom(\mathcal{S},\mathcal{O}(1))\oplus \Hom(\mathcal{O},\mathcal{O}(1))$
as $K$-vector spaces,
we have a natural homomorphism 
$\Hom(\mathcal{S},\mathcal{O}(1))\otimes_K P_1\to N$.
To show that this homomorphism is surjective, it is enough to show that the composition
\[\Hom(\mathcal{S},\mathcal{O}(1))\otimes \Hom(\mathcal{O},\mathcal{S})
\to \Hom(\mathcal{O},\mathcal{O}(1))\]
is surjectve. 
Note that 
the map sending $s$ to $s\wedge\bullet$ induces an isomorphism
\[\mathcal{S}\cong \calhom(\mathcal{S},\mathcal{O}(1)).\]
Hence
the composition above sends 
$(s\wedge\bullet)\otimes t$
to $s\wedge t$,
and 
it is surjective
by Lemma~\ref{wedgeproductOfS}.
Therefore 
the natural homomorphism 
$\Hom(\mathcal{S},\mathcal{O}(1))\otimes_K P_1
\to N$ is surjective.
Now the kernel of the 
surjection 
$\Hom(\mathcal{S},\mathcal{O}(1))\otimes_K P_1
\to N$ is isomorphic to the kernel $M$ of the 
composition above
as $K$-vector spaces,
and it is isomorphic to  
$M\otimes_K P_0$
as right $A$-modules.
Note here 
that $\dim_KM=11$.
Hence $S_2$ has the projective resolution as above.
Since $P_i\otimes_AG\cong G_i$, this implies that 
$S_2\lotimes_AG$ is nothing but the desired complex 
\[
0\to 
M\otimes \mathcal{O}
\to 
\Hom(\mathcal{S},\mathcal{O}(1))\otimes\mathcal{S}
\to \mathcal{O}(1)\to 0.
\]

Now we have the following exact sequence:
\[
0\to 
T_{\mathbb{P}^4}(-2)|_{\mathbb{Q}^3}
\to 
M\otimes \mathcal{O}
\to 
\Hom(\mathcal{S},\mathcal{O}(1))\otimes\mathcal{S}
\to \mathcal{O}(1)\to 0.
\]
Let  $\mathcal{G}$ be the kernel of the morphism
$\Hom(\mathcal{S},\mathcal{O}(1))\otimes\mathcal{S}
\to \mathcal{O}(1)$.
Then we have the following commutative diagram with exact rows:
\[
\xymatrix{
0\ar[r]&M\otimes\mathcal{O}\ar[d]\ar[r] &\Hom(\mathcal{S},\mathcal{O}(1))\otimes H^0(\mathcal{S})\otimes \mathcal{O} \ar[d]\ar[r] &H^0(\mathcal{O}(1))\otimes \mathcal{O}\ar[r]\ar[d]    & 0    \\
0\ar[r]&\mathcal{G} \ar[r]      &\Hom(\mathcal{S},\mathcal{O}(1))\otimes\mathcal{S}\ar[r]        &\mathcal{O}(1)\ar[r]              & 0    
}
\]
Note that the kernel of the middle vertical surjection is $\mathcal{S}^{\vee \oplus 4}$
by \eqref{SSdual}.
Since the kernel of the right vertical surjection is $\Omega_{\mathbb{P}^4}|_{\mathbb{Q}^3}$,
we obtain the desired exact sequence.
\end{proof}

\begin{lemma}\label{coevaluation}
On $\mathbb{Q}^3$, we have isomorphisms
\[\Hom(T_{\mathbb{P}^4}(-2)|_{\mathbb{Q}^3},\mathcal{S}^{\vee})
\cong H^0(\Omega_{\mathbb{P}^4}(2)|_{\mathbb{Q}^3}\otimes\mathcal{S}^{\vee})
\cong H^0(\Omega_{\mathbb{P}^4}(1)|_{\mathbb{Q}^3}\otimes\mathcal{S})
\cong 
\Hom(\mathcal{S}^{\vee},\Omega_{\mathbb{P}^4}(1)|_{\mathbb{Q}^3}),\]
and these vector spaces have dimension four.
Moreover we can regard the morphism $T_{\mathbb{P}^4}(-2)|_{\mathbb{Q}^3}
\to \mathcal{S}^{\vee \oplus 4}$ in \eqref{OneOfKeyExactSeq}
as the coevaluation morphism 
\[T_{\mathbb{P}^4}(-2)|_{\mathbb{Q}^3}
\to \mathcal{S}^{\vee}\otimes 
\Hom(T_{\mathbb{P}^4}(-2)|_{\mathbb{Q}^3},\mathcal{S}^{\vee})^{\vee}.\]
\end{lemma}
\begin{proof}
Since we have an exact sequence
\[
0\to \Omega_{\mathbb{P}^4}(2)|_{\mathbb{Q}^3}\to \mathcal{O}(1)\otimes H^0(\mathcal{O}(1))
\to \mathcal{O}(2)\to 0,
\]
we have the following exact sequence:
\[
0\to \Omega_{\mathbb{P}^4}(2)|_{\mathbb{Q}^3}\otimes\mathcal{S}^{\vee}
\to \mathcal{S}\otimes H^0(\mathcal{O}(1))
\to \mathcal{S}(1)\to 0.
\]
Note here that $\mathcal{S}$ is $0$-regular, since $h^q(\mathcal{S}(-i))=0$
for $i=1,2,3$ and all $q$.
Hence the natural map 
$H^0(\mathcal{S})\otimes H^0(\mathcal{O}(1))
\to H^0(\mathcal{S}(1))$ is surjective.
Moreover $h^0(\mathcal{S}(1))=16$ by 
\eqref{SobokuRRwithSpinor}.
Therefore $h^0(\Omega_{\mathbb{P}^4}(2)|_{\mathbb{Q}^3}\otimes\mathcal{S}^{\vee})=4$.

If the composite of the morphism $T_{\mathbb{P}^4}(-2)|_{\mathbb{Q}^3}
\to \mathcal{S}^{\vee \oplus 4}$ in \eqref{OneOfKeyExactSeq}
and some projection $\mathcal{S}^{\vee \oplus 4}\to \mathcal{S}^{\vee}$ is zero,
then $\Omega_{\mathbb{P}^4}(1)|_{\mathbb{Q}^3}$ admits $\mathcal{S}^{\vee}$ as a quotient,
which contradicts the fact that 
$\Hom (\Omega_{\mathbb{P}^4}(1)|_{\mathbb{Q}^3}, \mathcal{S}^{\vee})=0$.
Hence every element of 
$\Hom(T_{\mathbb{P}^4}(-2)|_{\mathbb{Q}^3},\mathcal{S}^{\vee})$
is obtained as the composite of 
$T_{\mathbb{P}^4}(-2)|_{\mathbb{Q}^3}\to \mathcal{S}^{\vee \oplus 4}$ 
and some projection $\mathcal{S}^{\vee \oplus 4}\to \mathcal{S}^{\vee}$.
\end{proof}

\begin{lemma}\label{NegativeQuotinetOnLineAri}
Let $\varphi:\mathcal{S}^{\vee}\to \Omega_{\mathbb{P}^4}(1)|_{\mathbb{Q}^3}$
be a morphism of $\mathcal{O}_{\mathbb{Q}^3}$-modules.
If $\varphi\neq 0$, then $\varphi$ is injective, and
there exists a line $L$ on $\mathbb{Q}^3$ such that
the restriction $\Coker(\varphi)|_L$ to $L$ of the cokernel $\Coker(\varphi)$ of $\varphi$ 
admits a negative degree quotient.
\end{lemma}
\begin{proof}
We have an exact sequence
\[0\to \Omega_{\mathbb{P}^4}(1)|_{\mathbb{Q}^3}
\xrightarrow{i} H^0(\mathcal{O}(1))\otimes \mathcal{O}\to \mathcal{O}(1)\to 0,\]
and the composite $i \circ\varphi$ can be written as 
\[i \circ\varphi
=\sum_{i=1}^r  l_i\otimes s_i^{\vee}\]
for some $l_i\in H^0(\mathcal{O}(1))$
and $s_i\in H^0(\mathcal{S})$,
where $s_i^{\vee}$ denotes the dual of the morphism $\mathcal{O}\to \mathcal{S}$
determined by $s_i$.
We may assume that $l_i\neq 0$ for all $i$.
By replacing $l_i$ if necessary, we may further assume that $s_1,\dots,s_r$ are linearly independent.
Since $h^0(\mathcal{S})=4$, we have $r\leq 4$.
Note that $\sum_{i=1}^r l_is_i^{\vee}=0$ in $\Hom (\mathcal{S}^{\vee},\mathcal{O}(1))$.
Hence $r\geq 2$. Moreover we have a surjective morphism 
\[\psi:\Coker(i\circ \varphi)\to \mathcal{O}(1).\]
Note that the morphism $\mathcal{O}^{\oplus r}\to \mathcal{S}$ determined by 
$(s_1,\dots,s_r)$ is generically surjective.
Hence we see that $i\circ \varphi$ is injective.
Therefore $\varphi$ is injective and 
\[\Coker(\varphi)\cong \Ker (\psi).\]

If $r=2$, then $\Coker(i\circ \varphi)\cong \mathcal{T}\oplus \mathcal{O}^{\oplus 3}$
for some torsion sheaf $\mathcal{T}$ on $\mathbb{Q}^3$.
Since $\mathcal{O}(1)$ is torsion-free, $\psi$ maps $\mathcal{T}$ to zero,
and we have a surjective morphism $\bar{\psi}:\mathcal{O}^{\oplus 3}\to \mathcal{O}(1)$.
On the other hand, $\bar{\psi}:\mathcal{O}^{\oplus 3}\to \mathcal{O}(1)$ cannot be surjecitve
since three hyperplane sections of $\mathbb{Q}^3$ always meet at a point. This is a contradiction.
Hence $r=3$ or $4$. Suppose that $r=4$. Then 
it follows from the exact sequence \eqref{SSdual}
that 
$\Coker(i\circ \varphi)\cong \mathcal{S}\oplus \mathcal{O}$.
Note that $\psi$ induces a morphism $\mathcal{S}\to \mathcal{O}(1)$,
which factors through $\mathcal{I}_L(1)$ 
for some line $L$ in $\mathbb{Q}^3$.
Since $L$ and a hyperplane in $\mathbb{Q}^3$ meet at a point, $\psi$ cannot be surjective.
Hence the case $r=4$ does not arise, and we have $r=3$.

Now 
it follows from the exact sequence \eqref{SSdual}
that the cokernel of the morphism 
determined by ${}^t(s_1^{\vee},s_2^{\vee},s_3^{\vee}):
\mathcal{S}^{\vee}\to \mathcal{O}^{\oplus 3}$
is isomorphic to the cokernel of some non-zero morphism 
$\mathcal{O}\to \mathcal{S}$, and hence it is isomorphic to $\mathcal{I}_M(1)$
for some line $M$ on $\mathbb{Q}^3$.
Therefore $\Coker(i\circ \varphi)\cong \mathcal{I}_M(1)\oplus \mathcal{O}^{\oplus 2}$,
and we have the following exact sequence:
\begin{equation}\label{CokerVarphiInc}
0\to \Coker(\varphi)\xrightarrow{i} \mathcal{I}_M(1)\oplus \mathcal{O}^{\oplus 2}\xrightarrow{\psi} \mathcal{O}(1)\to 0.
\end{equation}
Let $\mathbb{Q}^2$ be a general hyperplane section of $\mathbb{Q}^3$ containing $M$.
We may assume that $M$ is a divisor of type $(1,0)$ of $\mathbb{Q}^2$.
Then $\mathcal{I}_M(1)$ fits in the following exact sequence:
\[0\to \mathcal{O}_{\mathbb{Q}^3}\to \mathcal{I}_M(1)\to \mathcal{O}_{\mathbb{Q}^2}(0,1)\to 0.\]
By pulling back the sequence above 
to  a line $L$ of type $(0,1)$ in $\mathbb{Q}^2$, we obtain the following exact sequence:
\[
\mathcal{O}_L\to \mathcal{I}_M(1)\otimes \mathcal{O}_L\to \mathcal{O}_L\to 0.
\]
The image of $\mathcal{O}_L\to \mathcal{I}_M(1)\otimes \mathcal{O}_L$ is the torsion part of $\mathcal{I}_M(1)\otimes\mathcal{O}_L$.
Therefore $\psi\otimes 1_L$ factors through $\mathcal{O}_L^{\oplus 3}$ and induces
a surjection $\mathcal{O}_L^{\oplus 3}\to \mathcal{O}_L(1)$.
Hence $\Coker(\varphi)\otimes\mathcal{O}_L$ has $\mathcal{O}_L(-1)\oplus \mathcal{O}_L$ as a quotient.
\end{proof}

\begin{lemma}\label{CannotBeSurjective}
Let $\psi_1:T_{\mathbb{P}^4}(-2)|_{\mathbb{Q}^3}
\to \mathcal{S}^{\vee}$ be a morphism of $\mathcal{O}_{\mathbb{Q}^3}$-modules.
Then $\psi_1$ cannot be surjective. 
\end{lemma}
\begin{proof}
Suppose, to the contrary, that $\psi_1$ is surjective. 
Then $\mathcal{S}$ is a sub-bundle of $\Omega_{\mathbb{P}^4}(2)|_{\mathbb{Q}^3}$,
and consequently a sub-bundle of $\mathcal{O}(1)^{\oplus 5}$.
Since the quotient $\mathcal{O}(1)^{\oplus 5}/\mathcal{S}$ is a vector bundle,
it is isomorphic to $\mathcal{S}(1)\oplus \mathcal{O}(1)$
by \eqref{SSdual}, and the Euler sequence induces a surjection
$\mathcal{S}(1)\oplus \mathcal{O}(1)\to \mathcal{O}(2)$.
By taking the dual of the surjection and twisting it by $\mathcal{O}(2)$, 
we see that 
$\mathcal{O}$ can be a sub-bundle of $\mathcal{S}\oplus \mathcal{O}(1)$. 
On the other hand, a non-zero section $s$ of $H^0(\mathcal{S})$ defines a line $L$,
and $L$ and a hyperplane $H$ always meet at a point. Hence 
$\mathcal{O}$ cannot be a sub-bundle of $\mathcal{S}\oplus \mathcal{O}(1)$. 
This is a contradiction. Hence $\psi_1$ cannot be surjective. 
\end{proof}

Recall here that, for a coherent sheaf $\mathcal{F}$ of codimension $\geq p+1$
on a nonsingular projective variety $X$,
we have $c_i(\mathcal{F})=0$ for all $1\leq i\leq p$ (see, 
e.g.,
\cite[Example 15.3.6]{fl}).
Lemma~\ref{SupposedVeryImportant} below will be applied to $\psi_a$ in 
\eqref{E_2^-1,1resolutionInCaseh1E(-2)=1,c_3=0}
and \eqref{OneOfApplicationOfpsi},
and plays an crucial role in our proof of Theorem~\ref{Chern2}. 
\begin{lemma}\label{SupposedVeryImportant}
Let $\psi_a:T_{\mathbb{P}^4}(-2)|_{\mathbb{Q}^3}
\to \mathcal{S}^{\vee \oplus a}$ be a morphism of $\mathcal{O}_{\mathbb{Q}^3}$-modules
where $a$ is a positive integer. Then 
the restriction $\Coker(\psi_a)|_L$ of the cokernel $\Coker(\psi_a)$ of $\psi_a$ 
to some line  $L$ on $\mathbb{Q}^3$
admits a negative degree quotient.
In particular, if $\psi_1\neq 0$, then $\psi_1$ fits in the following exact sequence:
\[0\to \mathcal{O}(-1)^{\oplus 2}\to 
T_{\mathbb{P}^4}(-2)|_{\mathbb{Q}^3}\xrightarrow{\psi_1} \mathcal{S}^{\vee}
\to \mathcal{O}_L(-1)\to 0.\]
\end{lemma}
\begin{proof}
If the composite of the morphism 
$\psi_a$
and some projection $\mathcal{S}^{\vee\oplus a}\to \mathcal{S}^{\vee}$ is zero, then 
$\Coker(\psi_a)$
admits 
$\mathcal{S}^{\vee}$ as a quotient, 
and the assertion follows.
Hence we may assume that the composite 
cannot be zero for any projection 
$\mathcal{S}^{\vee\oplus a}\to \mathcal{S}^{\vee}$,
and this implies that $a\leq 4$
by Lemma~\ref{coevaluation}.

If $a=4$, then Lemma~\ref{S2nukelemma}
shows that $\Coker(\psi_4)\cong \Omega_{\mathbb{P}^4}(1)|_{\mathbb{Q}^3}$,
and the assertion follows.

If $a=3$, then 
$\psi_3$ 
can be regarded as the composite of the coevaluation morphism
\[\psi_4:T_{\mathbb{P}^4}(-2)|_{\mathbb{Q}^3}
\to \mathcal{S}^{\vee}\otimes 
\Hom(T_{\mathbb{P}^4}(-2)|_{\mathbb{Q}^3},\mathcal{S}^{\vee})^{\vee}\]
and some projection 
$\mathcal{S}^{\vee}\otimes 
\Hom(T_{\mathbb{P}^4}(-2)|_{\mathbb{Q}^3},\mathcal{S}^{\vee})^{\vee}
\to \mathcal{S}^{\vee\oplus 3}$.
Let $\mathcal{S}^{\vee}\to \mathcal{S}^{\vee\oplus 4}$ be the kernel of this projection,
and let $\varphi$ be the composite of the inclusion 
$\mathcal{S}^{\vee}\to \mathcal{S}^{\vee\oplus 4}$
and the surjection $ \mathcal{S}^{\vee\oplus 4}\to \Omega_{\mathbb{P}^4}(1)|_{\mathbb{Q}^3}$
in \eqref{OneOfKeyExactSeq}.
Then 
\begin{equation}\label{recall}
\Coker(\psi_3)\cong \Coker(\varphi)
\end{equation}
and $\Ker(\psi_3)\cong \Ker(\varphi)$
by the snake lemma.
Note that $\Hom(\mathcal{S}^{\vee},T_{\mathbb{P}^4}(-2)|_{\mathbb{Q}^3})=0$,
since $\Hom(\mathcal{S}^{\vee},\mathcal{O}(-1)^{\oplus 5})=0$
and $\Ext^1(\mathcal{S}^{\vee},\mathcal{O}(-2))=0$.
Hence $\varphi$ cannot be zero by 
\eqref{OneOfKeyExactSeq}.
Lemma~\ref{NegativeQuotinetOnLineAri} then shows that $\varphi$ is injective
and that the restriction $\Coker(\varphi)|_L$ to some line $L$ on $\mathbb{Q}^3$
admits a negative degree quotient.
Hence the assertion holds, and $\psi_3$ 
is injective.

Suppose that $a=2$. Then 
we can regard 
$\psi_2$ 
as the composite of some 
$\psi_3:T_{\mathbb{P}^4}(-2)|_{\mathbb{Q}^3}
\to \mathcal{S}^{\vee\oplus 3}$
and some projection 
$\mathcal{S}^{\vee\oplus 3}
\to \mathcal{S}^{\vee\oplus 2}$.
Let $\mathcal{S}^{\vee}\to \mathcal{S}^{\vee\oplus 3}$ be the kernel of this projection.
Note here that we have an exact sequence
\[
0\to 
T_{\mathbb{P}^4}(-2)|_{\mathbb{Q}^3}
\xrightarrow{\psi_3} \mathcal{S}^{\vee\oplus 3}
\to \Coker(\varphi)\to 0.
\]
Denote by $\varphi_1:\mathcal{S}^{\vee}\to \Coker(\varphi)$ the composite of the inclusion
$\mathcal{S}^{\vee}\to \mathcal{S}^{\vee\oplus 3}$
and the surjection $\mathcal{S}^{\vee\oplus 3}
\to \Coker(\varphi)$.
Then $\varphi_1$ cannot be zero, since 
$\Hom(\mathcal{S}^{\vee},T_{\mathbb{P}^4}(-2)|_{\mathbb{Q}^3})=0$. Moreover 
the snake lemma implies that 
\[\Coker(\psi_2)\cong \Coker(\varphi_1)\quad \textrm{and that }\quad 
\Ker(\psi_2)\cong \Ker(\varphi_1).\]

Recall the inclusion 
$i:\Coker(\varphi)\hookrightarrow \mathcal{I}_M(1)\oplus\mathcal{O}^{\oplus 2}$ 
in \eqref{CokerVarphiInc},
and consider the composite $i\circ\varphi_1$.
We have the following exact 
sequence:
\begin{equation}\label{imaruvarphi_1}
0\to \Coker(\varphi_1)\to \Coker(i\circ \varphi_1)\to \mathcal{O}(1)\to 0.
\end{equation}
Let $i\circ \varphi_1$ be equal to $(t^{\vee}, s_1^\vee, s_2^{\vee})$,
where $t^{\vee}\in \Hom(\mathcal{S}^{\vee},\mathcal{I}_M(1))$, $t\in H^0(\mathcal{S}(1))$,
$s_1^{\vee}$, $s_2^{\vee}
\in \Hom(\mathcal{S}^{\vee},\mathcal{O})$,
and $s_1$, $s_2\in H^0(\mathcal{S})$.
Since we have an exact sequence \eqref{CokerVarphiInc},
we have $t^{\vee}+h_1s_1^{\vee}+h_2s_2^{\vee}=0$ for some $h_1$, $h_2\in H^0(\mathcal{O}(1))$.
Now we have two cases:
\begin{enumerate}
\item $s_1$ and $s_2$ are linearly independent; 
\item $s_1$ and $s_2$ are linearly dependent.
\end{enumerate}
(1) If $s_1$ and $s_2$ are linearly independent, 
then $\varphi_1$ is injective,
and $\Coker(i\circ \varphi_1)$ has rank one. Thus we see that 
$\Coker(\varphi_1)$ is a torsion sheaf.
Moreover we claim that $\Coker(\varphi_1)$ is pure by \cite[Prop. 1.1.6]{Huybrechts-Lehn}:
first note that $\calext^q_{\mathbb{Q}^3}(\Coker(\varphi),\omega_{\mathbb{Q}^3})=0$ for all $q\geq2$;
thus $\calext^q_{\mathbb{Q}^3}(\Coker(\varphi_1),\omega_{\mathbb{Q}^3})=0$ for all $q\geq2$,
and hence $\Coker(\varphi_1)$ satisfies the generalized Serre's condition $S_{1,1}$
in \cite[Section 1.1]{Huybrechts-Lehn}. 
Now we compute the Chern polynomial of $\Coker(\varphi_1)$.
First note that $c_t(\Coker(\varphi))=c_t(\Omega_{\mathbb{P}^4}(1)|_{\mathbb{Q}^3})/c_t(\mathcal{S}^{\vee})
=1+lt^2-t^3$. 
Hence 
\[c_t(\Coker(\varphi_1))=c_t(\Coker(\varphi))/c_t(\mathcal{S}^{\vee})
=1+ht+2lt^2.\] 
Since $\Coker(\varphi_1)$ is a torsion sheaf, this implies that $\Coker(\varphi_1)$
is supported on a hyperplane section $H$ of $\mathbb{Q}^3$,
and the length of $\Coker(\varphi_1)$ at the generic point of $H$
is one.
Since $\Coker(\varphi_1)$ is pure, this implies that 
$\Coker(\varphi_1)$ is of the form $\mathcal{I}_{Z,H}(D)$,
where $D$ is a divisor on $H$ and 
$\mathcal{I}_{Z,H}$ denotes the ideal sheaf of 
some zero-dimensional closed subscheme $Z$
in $H$.
Note here that $c_t(\mathcal{O}_{H})=1+ht+2lt^2+2t^3$,
that $c_t(\mathcal{O}_L)=(c_t(\mathcal{S}^{\vee})/c_t(\mathcal{O}(-1)))^{-1}
=1-lt^2-t^3$,
and that $c_t(k(p))=1+2t^3$,
where $k(p)$ is the residue field at a point $p$
(see also 
\cite[Example 15.3.1]{fl} for the formula $c_t(k(p))=1+2t^3$).
Hence we see that $[D]=0\cdot l$ in $A^2\mathbb{Q}^3$.
Moreover if $D$ is of type $(d,-d)$, then 
$c_t(\mathcal{I}_{Z,H}(D))=1+ht+2lt^2+(2-2d^2-2\length Z)t^3$. 
Hence $(d,\length Z)=(0,1)$ or $(\pm 1,0)$.
Therefore $\Coker(\varphi_1)$ is isomorphic to 
either $\mathcal{I}_{p, H}$ or  $\mathcal{O}_H(d,-d)$ where $d=\pm 1$.
Thus the assertion holds.\\
(2) If $s_1$ and $s_2$ are linearly dependent, by replacing $s_i$ and $h_i$ if necessary, 
we may assume that $s_2=0$, and we have $t^{\vee}+h_1s_1^{\vee}=0$.
Set $\varphi_1':=(t^{\vee},s_1^{\vee}):
\mathcal{S}^{\vee}\to \mathcal{I}_M(1)\oplus \mathcal{O}_{\mathbb{Q}^3}$.
Then $\Coker(i\circ\varphi_1)\cong \Coker(\varphi_1')\oplus \mathcal{O}_{\mathbb{Q}^3}$
and $\Ker(\varphi_1)\cong \Ker(\varphi_1')$.
Note that $\varphi_1'\neq 0$ since $\varphi_1\neq 0$. Hence $s_1\neq 0$.
Let $L$ be the zero locus $(s_1)_0$ of $s_1$.
Then the composite of $\varphi_1'$ and the inclusion 
$\mathcal{I}_M(1)\oplus \mathcal{O}_{\mathbb{Q}^3}\to
\mathcal{O}(1)\oplus \mathcal{O}_{\mathbb{Q}^3}$
factors through the morphism 
$(-h_1,1):\mathcal{O}\to \mathcal{O}(1)\oplus \mathcal{O}_{\mathbb{Q}^3}$,
and we have the following commutative diagram with exact rows:
\begin{equation}\label{DecompOfvarphi'_1}
\xymatrix{
&\mathcal{S}^{\vee}\ar[d]_{\varphi_1'}\ar[r]^{s_1^{\vee}} 
&\mathcal{O}_{\mathbb{Q}^3}\ar[d]_{(-h_1,1)}\ar[r] 
&\mathcal{O}_{L}\ar[r]\ar[d]_{-\bar{h}_1}    & 0    \\
0\ar[r]&\mathcal{I}_M(1)\oplus\mathcal{O}_{\mathbb{Q}^3}\ar[r]
&\mathcal{O}(1)\oplus\mathcal{O}_{\mathbb{Q}^3}\ar[r] &\mathcal{O}_M(1)
\ar[r]              & 0    
}
\end{equation}
We see that 
$\im (\varphi_1')\cong \mathcal{I}_L$
and that 
$\Ker(\varphi_1')\cong \mathcal{O}(-1)$. 
We claim here that $\bar{h}_1\neq 0$.
Assume, to the contrary, that $\bar{h}_1=0$.
Then the snake lemma implies that $\Coker(\varphi_1')$ fits in the following exact sequence:
\[
0\to \mathcal{O}_L\to \Coker(\varphi_1')\to \mathcal{O}(1)\to \mathcal{O}_M(1)\to 0. 
\]
Since $\mathcal{O}_L$ is a torsion sheaf, the surjection 
$\Coker(\varphi_1')\oplus \mathcal{O}_{\mathbb{Q}^3}\to \mathcal{O}(1)$
induces a surjection $\mathcal{I}_M(1)\oplus \mathcal{O}_{\mathbb{Q}^3}\to \mathcal{O}(1)$.
On the other hand, the morphism 
$\mathcal{I}_M(1)\oplus \mathcal{O}_{\mathbb{Q}^3}\to \mathcal{O}(1)$
cannot be surjective since a line $M$ and a hyperplane meets at least at one point.
This is a contradiction. Hence $\bar{h}_1\neq 0$, and thus $L=M$. 
Moreover 
the commutative diagram \eqref{DecompOfvarphi'_1} induces 
the following exact sequence by the snake lemma:
\[0\to \Coker(\varphi_1')\to \mathcal{O}(1)\to k(p)\to 0,\]
where $p=(\bar{h}_1)_0$.
Therefore $\Coker(\varphi_1')=\mathcal{I}_p(1)$.
The exact sequence \eqref{imaruvarphi_1}, i.e., the sequence
\[0\to \Coker(\varphi_1)\to \mathcal{I}_p(1)\oplus \mathcal{O}_{\mathbb{Q}^3}
\to \mathcal{O}(1)\to 0\]
then shows that $\Coker(\varphi_1)=\mathcal{I}_p$.
Thus the assertion also holds if $s_1$ and $s_2$ are linearly dependent.

Suppose that $a=1$.
Then 
$\psi_1$ 
can be regarded as the composite of some 
$\psi_2:T_{\mathbb{P}^4}(-2)|_{\mathbb{Q}^3}
\to \mathcal{S}^{\vee\oplus 2}$
and some projection 
$\mathcal{S}^{\vee\oplus 2}
\to \mathcal{S}^{\vee}$.
Let $\mathcal{S}^{\vee}\to \mathcal{S}^{\vee\oplus 2}$ be the kernel of this projection.
Note here that we have an exact sequence
\[
0\to \Ker(\varphi_1)\to 
T_{\mathbb{P}^4}(-2)|_{\mathbb{Q}^3}
\xrightarrow{\psi_2} \mathcal{S}^{\vee\oplus 2}
\to \Coker(\varphi_1)\to 0.
\]
Denote by $\varphi_2:\mathcal{S}^{\vee}\to \Coker(\varphi_1)$ the composite of the inclusion
$\mathcal{S}^{\vee}\to \mathcal{S}^{\vee\oplus 2}$
and the surjection $\mathcal{S}^{\vee\oplus 2}
\to \Coker(\varphi_1)$.
Then the snake lemma implies that 
\[\Coker(\psi_1)\cong \Coker(\varphi_2).\]
We claim here that $\varphi_2\neq 0$.
Indeed, if $\Ker(\varphi_1)=0$, then the statement holds
since $\Hom(\mathcal{S}^{\vee},T_{\mathbb{P}^4}(-2)|_{\mathbb{Q}^3})=0$.
If $\Ker(\varphi_1)\neq 0$, 
then $\Ker(\varphi_1)\cong \Ker(i\circ\varphi_1)\cong \mathcal{O}(-1)$,
and we have $\Hom(\mathcal{S}^{\vee},\im (\psi_2))=0$,
since $\Ext^1(\mathcal{S}^{\vee},\mathcal{O}(-1))=0$.
Hence 
$\varphi_2\neq 0$.
As we have seen above, $\Coker(\varphi_1)$ is one of the following:
$\mathcal{I}_{p,H}$; $\mathcal{O}_{H}(d,-d)$ where $d=\pm 1$;
$\mathcal{I}_p$.
If $\Coker(\varphi_1)=\mathcal{I}_{p,H}$, 
$\Coker(\varphi_1)$ is contained in $\mathcal{O}_{H}$.
Since $\Ext^1(\mathcal{S}^{\vee},\mathcal{O}(-1))=0$, the restriction map 
$\Hom (\mathcal{S}^{\vee},\mathcal{O}_{\mathbb{Q}^3})
\to \Hom(\mathcal{S}^{\vee},\mathcal{O}_{H})$ is surjective.
Hence we see that $\varphi_2:\mathcal{S}^{\vee}\to 
\mathcal{I}_{p,H}$ factors through
the ideal sheaf 
$\mathcal{I}_L$ of some line $L$ passing through $p$.
If $L$ is not contained in $H$, then $L\cap H=\{p\}$ 
and we have a surjection $\mathcal{I}_L\to \mathcal{I}_{p,H}$.
Thus $\varphi_2$ is surjective. Hence $\Coker(\psi_1)=0$. This contradicts 
Lemma~\ref{CannotBeSurjective}. Therefore $L$ is contained in $H$,
and the morphism $\mathcal{I}_L\to \mathcal{I}_{p,H}$ induces
an inclusion $\mathcal{I}_{L,H}\hookrightarrow 
\mathcal{I}_{p,H}$.
Thus $\Coker(\varphi_2)\cong \mathcal{O}_L(-p)$.
If $\Coker(\varphi_1)=\mathcal{O}_{H}(d,-d)$ with $d=\pm 1$,
then $\varphi_2$ factors through a morphism 
$\bar{\varphi}_2:\mathcal{O}(-1,0)\oplus \mathcal{O}(0,-1)
\to  \mathcal{O}_{H}(d,-d)$.
Note that $\bar{\varphi}_2\neq 0$, since $\varphi_2\neq 0$.
Thus 
$\Coker(\varphi_2)$ is 
$\mathcal{O}_L(-1)$ for some line $L$ on $\mathbb{Q}^3$.
Hence the assertion follows. 
If $\Coker(\varphi_1)=\mathcal{I}_p$,
then $\Coker(\varphi_2)\cong \mathcal{O}_L(-p)$,
and the assertion follows.

Finally we claim that $\Ker(\psi_1)\cong \mathcal{O}(-1)^{\oplus 2}$. 
We have the following commutative 
diagram with exact rows:
\[
\xymatrix{
0\ar[r]&0\ar[d]\ar[r] &T_{\mathbb{P}^4}(-2)|_{\mathbb{Q}^3}\ar[d]_{\psi_2}\ar[r] 
&T_{\mathbb{P}^4}(-2)|_{\mathbb{Q}^3}\ar[r]\ar[d]_{\psi_1}    & 0    \\
0\ar[r]&\mathcal{S}^{\vee} \ar[r]      &\mathcal{S}^{\vee \oplus 2}\ar[r] &\mathcal{S}^{\vee}
\ar[r]              & 0    
}
\]
By the snake lemma, we have the following exact sequence:
\[0\to \Ker(\varphi_1)\to \Ker(\psi_1)\to \mathcal{S}^{\vee}
\xrightarrow{\varphi_2} \Coker(\varphi_1)\to \mathcal{O}_L(-1)\to 0.\]
Now $(\Ker(\varphi_1),\Coker(\varphi_1))$ is one of the following:
$(0,\mathcal{I}_{p,H})$;
$(0,\mathcal{O}_{H}(d,-d))$ where $d=\pm 1$;
$(\mathcal{O}(-1),\mathcal{I}_p)$.
If $(\Ker(\varphi_1),\Coker(\varphi_1))=(0,\mathcal{I}_{p,H})$,
then $\Ker(\psi_1)=\Ker(\varphi_2)$, $\im \varphi_2=\mathcal{I}_{L,H}$,
and the surjection $\mathcal{S}^{\vee}\to \mathcal{I}_{L,H}$ factors through $\mathcal{I}_L$.
We have the following commutative diagram with exact rows:
\[
\xymatrix{
0\ar[r]&\Ker(\varphi_2)\ar[d]\ar[r] &\mathcal{S}^{\vee}\ar[d]\ar[r] 
&\mathcal{I}_{L,H}\ar[r]\ar@{=}[d]    & 0    \\
0\ar[r]&\mathcal{O}(-1) \ar[r]      &\mathcal{I}_L\ar[r] &\mathcal{I}_{L,H}
\ar[r]              & 0    
}
\]
Hence we have the following exact sequence:
\begin{equation}\label{KurikaesiWoSakeru}
0\to \mathcal{O}(-1)\to \Ker(\varphi_2)\to \mathcal{O}(-1)\to 0.
\end{equation}
Therefore the claim holds in this case. 
If $(\Ker(\varphi_1),\Coker(\varphi_1))=(0,\mathcal{O}_{H}(d,-d))$,
we may assume, without loss of generality, that $d=1$.
Then $\Ker(\psi_1)=\Ker(\varphi_2)$, $\im \varphi_2=\mathcal{O}_{H}(0,-1)$.
Take an inclusion $\mathcal{O}_H(0,-1)\hookrightarrow \mathcal{O}_H$ and
we regard $\mathcal{O}_H(0,-1)$ as the ideal sheaf $\mathcal{I}_{M,H}$ of some line $M$ in $H$. 
The dual of the composite of 
the projection
$\mathcal{O}_H(-1,0)\oplus\mathcal{O}_H(0,-1)\to \mathcal{O}_H(0,-1)$
and the inclusion $\mathcal{O}_H(0,-1)\hookrightarrow \mathcal{O}_H$ gives rise a section
$\bar{s}$ of $H^0(\mathcal{O}_H(1,0)\oplus \mathcal{O}_H(0,1))$.
Since $H^0(\mathcal{S})\cong H^0(\mathcal{O}_H(1,0)\oplus \mathcal{O}_H(0,1))$,
we can lift the section $\bar{s}$ to a section $s$ of $H^0(\mathcal{S})$.
Then $(s)_0=M$ and the surjection 
$\mathcal{S}^{\vee}\to \im \varphi_2=\mathcal{O}_{H}(0,-1)$ factors through $\mathcal{I}_M$.
We have the following commutative diagram with exact rows:
\[
\xymatrix{
0\ar[r]&\Ker(\varphi_2)\ar[d]\ar[r] &\mathcal{S}^{\vee}\ar[d]\ar[r] 
&\mathcal{O}_{H}(0,1)\ar[r]\ar@{=}[d]    & 0    \\
0\ar[r]&\mathcal{O}(-1) \ar[r]      &\mathcal{I}_M\ar[r] &\mathcal{I}_{M,H}
\ar[r]              & 0    
}
\]
Hence we obtain the exact sequence~\eqref{KurikaesiWoSakeru}.
Therefore the claim holds in this case. 
If $(\Ker(\varphi_1),\Coker(\varphi_1))=(\mathcal{O}(-1),\mathcal{I}_p)$,
then $\im (\varphi_2)=\mathcal{I}_L$, and we have the following exact sequence:
\[0\to \mathcal{O}(-1)\to \Ker(\psi_1)\to \mathcal{O}(-1)\to 0.\]
Hence the claim holds in this case too. 
\end{proof}

Lemma~\ref{SupposedVeryImportantPart2} will be applied to $\pi$
in \eqref{oneOfapplicationOfpi}, and plays an crucial role in the proof
of Theorem~\ref{Chern2}.
\begin{lemma}\label{SupposedVeryImportantPart2}
Let $\psi_a:T_{\mathbb{P}^4}(-2)|_{\mathbb{Q}^3}
\to \mathcal{S}^{\vee \oplus a}$ be a morphism of $\mathcal{O}_{\mathbb{Q}^3}$-modules
where $a$ is  a positive integer,
and let $\pi:\mathcal{O}_{\mathbb{Q}^3}(-1)\to \Coker(\psi_a)$ be a 
morphism of $\mathcal{O}_{\mathbb{Q}^3}$-modules.
If $\Coker(\pi)$ does not admit a negative degree quotient,
then $a=1$, $\Coker(\pi)=0$ and $\Ker(\pi)$ is isomorphic to $\mathcal{I}_L(-1)$
for some line $L$ in $\mathbb{Q}^3$. 
\end{lemma}
\begin{proof}
We may assume that $\pi\neq 0$.

Suppose that $\Coker(\psi_a)$ admits $\mathcal{S}^{\vee}$ as a quotient;
let $p:\Coker(\psi_a)\to \mathcal{S}^{\vee}$ be the surjection.
Note that $\Coker(\pi)$ admits $\Coker(p\circ\pi)$ as a quotient. 
If $p\circ \pi=0$, then $\Coker(p\circ \pi)\cong \mathcal{S}^{\vee}$,
and if $p\circ \pi\neq 0$, then $\Coker(p\circ \pi)\cong \mathcal{I}_L$ for some line 
$L$ in $\mathbb{Q}^3$. Therefore the restriction of $\Coker(\pi)$ to a line
admits a negative degree quotient. 

In the following, we assume that 
$\Coker(\psi_a)$ does not admit $\mathcal{S}^{\vee}$ as a quotient.
Hence 
$a\leq 4$ by Lemma~\ref{coevaluation}.

Suppose that $a=4$. Then $\Coker(\psi_4)\cong \Omega_{\mathbb{P}^4}(1)|_{\mathbb{Q}^3}$
by Lemmas~\ref{S2nukelemma} and \ref{coevaluation}.
Since $\Omega_{\mathbb{P}^4}(1)|_{L}\cong \mathcal{O}_L(-1)\oplus\mathcal{O}_L^{\oplus 3}$
for any line $L$ in $\mathbb{Q}^3$,
if $\Coker(\pi)|_L$ does not admit 
a negative degree quotient for any line $L$ in  $\mathbb{Q}^3$,
we see that $\Coker(\pi)|_L\cong \mathcal{O}_L^{\oplus 3}$ for any line $L$ in  $\mathbb{Q}^3$. 
This implies that $\Coker(\pi)\cong \mathcal{O}_{\mathbb{Q}^3}^{\oplus 3}$ 
by \cite[(3.6.1) Lemma]{w3}.
Thus $\Omega_{\mathbb{P}^4}(1)|_{\mathbb{Q}^3}
\cong \mathcal{O}(-1)\oplus \mathcal{O}^{\oplus 3}$,
which contradicts $H^0(\Omega_{\mathbb{P}^4}(1)|_{\mathbb{Q}^3})=0$.
Therefore $\Coker(\pi)|_L$ admits a negative degree quotient 
for some line $L$ in  $\mathbb{Q}^3$.

Suppose that $a=3$.
Recall that $\Coker(\psi_3)\cong \Coker(\varphi)$
in \eqref{recall}.
Recall also 
the inclusion $i:\Coker(\varphi)\hookrightarrow \mathcal{I}_M(1)\oplus\mathcal{O}^{\oplus 2}$ 
in \eqref{CokerVarphiInc},
and consider the composite $i\circ\pi$.
We have the following exact 
sequence:
\begin{equation}\label{DefOfRho}
0\to \Coker(\pi)\to \Coker(i\circ \pi)
\xrightarrow{\rho}
\mathcal{O}(1)\to 0.
\end{equation}
Let $i\circ \pi$ be equal to $(t, g_1, g_2)$,
where $t\in \Hom(\mathcal{O}(-1),\mathcal{I}_M(1))\cong H^0(\mathcal{I}_M(2))$,
$g_1$, $g_2
\in \Hom(\mathcal{O}(-1),\mathcal{O})\cong H^0(\mathcal{O}(1))$.
Since we have an exact sequence \eqref{CokerVarphiInc},
we have $t+h_1g_1+h_2g_2=0$ for some $h_1$, $h_2\in H^0(\mathcal{O}(1))$.
Now we have two cases:
\begin{enumerate}
\item $g_1$ and $g_2$ are linearly independent; 
\item $g_1$ and $g_2$ are linearly dependent.
\end{enumerate}
(1) If $g_1$ and $g_2$ are linearly independent, 
then the cokernel of the morphism $(g_1,g_2):\mathcal{O}(-1)\to \mathcal{O}^{\oplus 2}$
is of the form $\mathcal{I}_C(1)$, where $C$ is the conic defined by $g_1$ and $g_2$.
Hence $\Coker(i\circ \pi)$ fits in the following exact sequence:
\[0\to \mathcal{I}_M(1)\to \Coker(i\circ \pi)\to \mathcal{I}_C(1)\to 0.\] 
Now consider the composite of the injection $\mathcal{I}_M\to \Coker(i\circ\pi)(-1)$
and the surjection $\rho (-1):\Coker(i\circ\pi)(-1)\to \mathcal{O}$.
The composite is nothing but the inclusion $\mathcal{I}_M\hookrightarrow \mathcal{O}$
and its cokernel is $\mathcal{O}_M$.
Thus the surjection 
$\rho (-1)$
induces a surjection 
$\bar{\rho}(-1):\mathcal{I}_C\to \mathcal{O}_M$. This implies that $C\cap M=\emptyset$.
Moreover 
$\Coker(\pi)(-1)\cong \Ker (\bar{\rho}(-1))
\cong \mathcal{I}_{C\sqcup M}$.
Hence $\Coker(\pi)\cong \mathcal{I}_{C\sqcup M}(1)$.
Note that the conic $C$ and the line $M$ can be joined by a line 
$L$ in $\mathbb{Q}^3$.
Indeed, any hyperplane section $H$ containing $M$ intersects $C$ at some point $p$,
and the point $p$ and $M$ can be joined by a line $L$ in $H$.
Now we see that 
$\Coker(\pi)|_L$ admits a negative degree quotient. \\
(2) If $g_1$ and $g_2$ are linearly dependent,
by replacing $g_i$ and $h_i$ if necessary, 
we may assume that $g_2=0$, and we have $t+h_1g_1=0$.
Set $\pi_1':=(t,g_1):
\mathcal{O}(-1)\to \mathcal{I}_M(1)\oplus \mathcal{O}_{\mathbb{Q}^3}$.
Then $\Coker(i\circ\pi)\cong \Coker(\pi')\oplus \mathcal{O}_{\mathbb{Q}^3}$.
Note that $\pi'\neq 0$ since $\pi\neq 0$. Hence $g_1\neq 0$.
Let $H$ be the hyperplane defined by $g_1$.
Then 
we have the following commutative diagram with exact rows:
\begin{equation}\label{DecompOfpi'}
\xymatrix{
0\ar[r]&\mathcal{O}(-1)\ar[d]_{\pi'}\ar[r]^{g_1} &\mathcal{O}_{\mathbb{Q}^3}\ar[d]_{(-h_1,1)}\ar[r] 
&\mathcal{O}_{H}\ar[r]\ar[d]_{-\bar{h}_1}    & 0    \\
0\ar[r]&\mathcal{I}_M(1)\oplus\mathcal{O}_{\mathbb{Q}^3}\ar[r]
&\mathcal{O}(1)\oplus\mathcal{O}_{\mathbb{Q}^3}\ar[r] &\mathcal{O}_M(1)
\ar[r]              & 0    
}
\end{equation}
We claim here that $\bar{h}_1\neq 0$.
Assume, to the contrary, that $\bar{h}_1=0$.
Then  the snake lemma shows that 
we have 
the following exact sequence:
\[
0\to \mathcal{O}_H\to \Coker(\pi')\to 
\mathcal{O}(1)\to \mathcal{O}_M(1)\to 0. 
\]
Since $\mathcal{O}_H$ is a torsion sheaf, the surjection 
$\rho:\Coker(\pi')\oplus \mathcal{O}_{\mathbb{Q}^3}\to \mathcal{O}(1)$
sends $\mathcal{O}_H$ to zero, and thus $\rho$ 
induces a surjection $\mathcal{I}_M(1)\oplus \mathcal{O}_{\mathbb{Q}^3}\to \mathcal{O}(1)$.
On the other hand, the morphism 
$\mathcal{I}_M(1)\oplus \mathcal{O}_{\mathbb{Q}^3}\to \mathcal{O}(1)$
cannot be surjective since a line $M$ and a hyperplane meets at least at one point.
This is a contradiction. Hence $\bar{h}_1\neq 0$. 
Then 
the kernel of the morphism $-\bar{h}_1:\mathcal{O}_H\to \mathcal{O}_M(1)$
is $\mathcal{O}_H(-M)$ and the cokernel of $-\bar{h}_1$ is $k(p)$ 
for some point $p\in M$.
Hence the commutative diagram \eqref{DecompOfpi'} 
induces 
the following exact sequence by the snake lemma:
\[0\to \mathcal{O}_H(-M)\to \Coker(\pi')\to \mathcal{O}(1)\to k(p)\to 0.\]
Since $\mathcal{O}_H(-M)$ is a torsion sheaf, the surjection 
$\rho:\Coker(\pi')\oplus \mathcal{O}_{\mathbb{Q}^3}\to \mathcal{O}(1)$
sends $\mathcal{O}_H(-M)$ to zero, and thus 
the inclusion $\mathcal{O}_H(-M)\hookrightarrow 
\Coker(\pi')\oplus\mathcal{O}_{\mathbb{Q}^3}$ induces an 
inclusion $\mathcal{O}_H(-M)\hookrightarrow \Coker(\pi)$.
The exact sequence \eqref{DefOfRho}
induces the following exact sequence: 
\[0\to \Coker(\pi)/
\mathcal{O}_H(-M)
\to \mathcal{I}_p(1)\oplus \mathcal{O}_{\mathbb{Q}^3}
\to \mathcal{O}(1)\to 0.\]
This shows that $\Coker(\pi)/\mathcal{O}_H(-M)=\mathcal{I}_p$.

Suppose that $a=2$. As we have seen in the proof of Lemma~\ref{SupposedVeryImportant},
$\Coker(\psi_2)$ is isomorphic to $\Coker(\varphi_1)$,
and $\Coker(\varphi_1)$ is one of the following: $\mathcal{I}_{p,H}$;
$\mathcal{O}_H(d,-d)$ where $d=\pm 1$;
$\mathcal{I}_p$.
If $\Coker(\varphi_1)=\mathcal{I}_{p,H}$,
then 
$\Coker(\pi)$ admits
$\mathcal{O}_C(-p)$ as a quotient, where $C$ is a conic on $H$.
If $\Coker(\varphi_1)=\mathcal{O}_{H}(d,-d)$ with $d=\pm 1$,
then 
$\Coker(\pi)$ admits
$\mathcal{O}_L(-1)$ as a quotient, where $L$ is a line on $H$.
If $\Coker(\varphi_1)=\mathcal{I}_{p}$,
then $\Coker(\pi)$
admits $\mathcal{I}_{p,H}$ as a quotient. 
Hence the assertion follows if $a=2$.

Suppose that $a=1$.
As we have seen in the proof of Lemma~\ref{SupposedVeryImportant},
$\Coker(\psi_1)$ is isomorphic to $\Coker(\varphi_2)$,
and $\Coker(\varphi_2)$ is isomorphic to $\mathcal{O}_L(-p)$.
Since $\pi\neq 0$, $\pi:\mathcal{O}(-1)\to \mathcal{O}_L(-p)$ is surjective,
and $\Ker(\pi)\cong \mathcal{I}_L(-1)$. This completes the proof.
\end{proof}

\section{A lower bound for the third Chern class}\label{lowerBound}
Note that 
\begin{equation}\label{tautologicalSelfIntersectionNonNegative}
c_3\geq 2c_1c_2-c_1^3
\end{equation}
for a nef vector bundle $\mathcal{E}$ on a complete $3$-fold $X$,
since $H(\mathcal{E})^{r+2}
=c_3-2c_1c_2+c_1^3\geq 0$
for a nef line bundle $H(\mathcal{E})$.
If there exists 
an injection $\mathcal{L}
\to 
\mathcal{E}$ from a line bundle $\mathcal{L}$,
then we have a 
lower bound,
which is better if $\mathcal{L}\cong \mathcal{O}(D)$ for some effective divisor $D$,
as the following lemma shows:
\begin{lemma}\label{c3lowerbound}
Let $\mathcal{E}$ be a nef vector bundle of rank $r$ on a 
complete variety $X$ of dimension three.
Let $\mathcal{L}$ be a line bundle on $X$
such that $H^0(\mathcal{E}\otimes \mathcal{L}^{-1})\neq 0$.
Then we have the following inequality:
\[c_3\geq 2c_1c_2-c_1^3+
(c_1^2-c_2)c_1(\mathcal{L}).\]
\end{lemma}

\begin{proof}
Let $\pi:\tilde{X}\to X$ be 
a proper birational morphism with $\tilde{X}$ smooth.
Then $H^0(\pi^*\mathcal{E}\otimes \pi^*\mathcal{L}^{-1})\neq 0$,
and the inequality for the inclusion  $\pi^*\mathcal{L}\to \pi^*\mathcal{E}$
implies the desired inequality by the projection formula.
Hence we may assume that $X$ is smooth.

Let $s$  be a non-zero element of 
$H^0(\mathcal{E}\otimes \mathcal{L}^{-1})$.
If the zero locus $(s)_0$ of $s$ contains a divisor $D_1$,
then we may replace 
$\mathcal{L}$ by $\mathcal{L}\otimes \mathcal{O}(D_1)$
since 
$(c_1^2-c_2)c_1(\mathcal{O}(D_1))\geq 0$.
Hence we may assume that the zero locus $(s)_0$ of $s$ has codimension $\geq 2$.
Set $Z:=(s)_0$,
and let $\pi:Y\to X$ be the blowing-up 
along the closed subscheme $Z$.
Denote by $E$ the exceptional divisor of $\pi$. 
Then we have the following exact sequence:
\[0\to \pi^*\mathcal{L}\otimes \mathcal{O}(E)\to \pi^*\mathcal{E}
\to \mathcal{G}\to 0,\]
where $\mathcal{G}$ a nef vector bundle of rank $r-1$ on $Y$.
Hence we have 
\[0\leq H(\mathcal{G})^{r+1}=c_3(\mathcal{G})
-2c_1(\mathcal{G})c_2(\mathcal{G})+c_1(\mathcal{G})^3.\]
Since 
\[
\begin{split}
c_t(\mathcal{G})
&=\dfrac{c_t(\pi^*\mathcal{E})}{c_t(\pi^*\mathcal{L}\otimes \mathcal{O}(E))}\\
&=
(1+c_1(\pi^*\mathcal{E})t+c_2(\pi^*\mathcal{E})t^2+
c_3(\pi^*\mathcal{E})t^3)
\\
&\qquad \times
(1-c_1(\pi^*\mathcal{L}\otimes\mathcal{O}(E))t
+c_1(\pi^*\mathcal{L}\otimes\mathcal{O}(E))^2t^2
-c_1(\pi^*\mathcal{L}\otimes\mathcal{O}(E))^3t^3),
\end{split}\]
we have 
\[
\begin{split}
c_1(\mathcal{G})
&=c_1(\pi^*\mathcal{E})-c_1(\pi^*\mathcal{L}\otimes \mathcal{O}(E));\\
c_2(\mathcal{G})
&=c_2(\pi^*\mathcal{E})
-c_1(\pi^*\mathcal{E})c_1(\pi^*\mathcal{L}\otimes \mathcal{O}(E))
+c_1(\pi^*\mathcal{L}\otimes \mathcal{O}(E))^2;\\
c_3(\mathcal{G})
&=
c_3(\pi^*\mathcal{E})
-c_2(\pi^*\mathcal{E})c_1(\pi^*\mathcal{L}\otimes \mathcal{O}(E))
+c_1(\pi^*\mathcal{E})c_1(\pi^*\mathcal{L}\otimes \mathcal{O}(E))^2\\
&\quad
-c_1(\pi^*\mathcal{L}\otimes \mathcal{O}(E))^3.
\end{split}
\]
Hence we infer that 
\[
\begin{split}
H(\mathcal{G})^{r+1}&=
c_3(\pi^*\mathcal{E})
-2c_1(\pi^{*}\mathcal{E})c_2(\pi^*\mathcal{E})
+c_1(\pi^*\mathcal{E})^3
\\
&\quad
+c_2(\pi^*\mathcal{E})c_1(\pi^*\mathcal{L}\otimes \mathcal{O}(E))
-c_1(\pi^*\mathcal{E})^2c_1(\pi^*\mathcal{L}\otimes \mathcal{O}(E)).
\end{split}\]
Note here that $c_2(\pi^*\mathcal{E})c_1(\mathcal{O}(E))=0$
and that $c_1(\pi^*\mathcal{E})^2c_1(\mathcal{O}(E))=0$,
since $\dim Z\leq 1$.
Therefore we have 
\[
0\leq
c_3(\pi^*\mathcal{E})
-2c_1(\pi^{*}\mathcal{E})c_2(\pi^*\mathcal{E})
+c_1(\pi^*\mathcal{E})^3
+c_2(\pi^*\mathcal{E})c_1(\pi^*\mathcal{L})
-c_1(\pi^*\mathcal{E})^2c_1(\pi^*\mathcal{L}).
\]
Hence we obtain the desired inequality by the projection formula.
\end{proof}
Lemma~\ref{c3lowerbound} will be applied to $\mathcal{E}$ 
in Section~\ref{Case6h^0(E(-1))=1}.

\section{Set-up for the proof of Theorem~\ref{Chern2}}\label{Set-up for the case $n=3$}
Let $\mathcal{E}$ be a nef vector bundle of rank $r$ on $\mathbb{Q}^3$
with $c_1=2h$.
It follows from \cite[Lemma~4.1 (1)]{MR4453350} that
\begin{equation}\label{firstvanishing}
h^q(\mathcal{E}(t))=0 \textrm{ for } q>0 \textrm{ and } t\geq 0.
\end{equation}
Moreover, if $H(\mathcal{E})^{r+2}
=c_3-2c_1c_2+c_1^3
=c_3-4c_2h+16>0$,
then 
\begin{equation}\label{higherE(-1)vanish}
h^q(\mathcal{E}(-1))=0 \textrm{ for } q>0
\end{equation}
by \cite[Lemma~4.1 (2)]{MR4453350}.
Note here that 
\begin{equation}\label{c3nonnegative}
c_3\geq 0
\end{equation}
by \cite[Theorem~8.2.1]{MR2095472},
since $\mathcal{E}$ is nef.
Hence we see that 
\begin{equation}\label{c2h<4vanishing}
h^q(\mathcal{E}(-1))=0 \textrm{ for } q>0
\textrm{ if }
c_2h\leq 3.
\end{equation}
It follows from \cite[Lemma~4.3]{MR4453350} that
\begin{equation}\label{firstvanishingforS}
\Ext^q(\mathcal{S},\mathcal{E}(2))=0 \textrm{ for } q>0.
\end{equation}
The exact sequence~\eqref{SSdual} together with the isomorphism~\eqref{canonicalIsom} 
implies that 
$\mathcal{S}^{\vee}\otimes \mathcal{E}(2)$ fits in an exact sequence
\[
0\to \mathcal{S}^{\vee}\otimes \mathcal{E}(1)\to 
\mathcal{E}(1)^{\oplus 4}\to \mathcal{S}^{\vee}\otimes\mathcal{E}(2)\to 0.
\]
It then follows from \eqref{firstvanishing} and \eqref{firstvanishingforS}
that 
\begin{equation}\label{ExtSE(1)q2ijouvanishing}
\Ext^q(\mathcal{S},\mathcal{E}(1))=0 \textrm{ for } q\geq 2.
\end{equation}

If $h^0(\mathcal{E}(-2))\neq 0$,
then $\mathcal{E}\cong \mathcal{O}(2)\oplus \mathcal{O}^{\oplus r-1}$ by 
\cite[Proposition 5.1 and Remark 5.3]{MR4453350}.
Thus we will always assume 
that 
\begin{equation}\label{h^0(E(-2))vanish}
h^0(\mathcal{E}(-2))=0
\end{equation}
in the following.
It follows from Theorem~\ref{Chern(2,2)} that 
\begin{equation}
h^q(\mathcal{E}|_{\mathbb{Q}^2})=0 \textrm{ for } q\geq 2.
\end{equation}
Moreover
\begin{equation}
h^1(\mathcal{E}|_{\mathbb{Q}^2})=
\begin{cases}
1 &\textrm{if }\mathcal{E}|_{\mathbb{Q}^2}\textrm{ belongs to Case (11) of 
Theorem~\ref{Chern(2,2)};}\\
0 &\textrm{otherwise.}
\end{cases}
\end{equation}
The vanishing~\eqref{firstvanishing} then shows that
\begin{equation}\label{h3E(-1)vanish}
h^3(\mathcal{E}(-1))=0.
\end{equation}
Moreover 
\begin{equation}\label{h2E(-1)vanish}
h^2(\mathcal{E}(-1))=0
\textrm{ unless }\mathcal{E}|_{\mathbb{Q}^2}\textrm{ belongs to Case }(11)\textrm{ of 
Theorem~\ref{Chern(2,2)}.}
\end{equation}
It follows from Theorem~\ref{Chern(2,2)} that 
\begin{equation}
h^q(\mathcal{E}(-1)|_{\mathbb{Q}^2})=0 \textrm{ for } q\geq 2.
\end{equation}
The vanishing~\eqref{h3E(-1)vanish} then shows that
\begin{equation}\label{h3E(-2)vanish}
h^3(\mathcal{E}(-2))=0.
\end{equation}
The exact sequence \eqref{SSdual} together with \eqref{canonicalIsom}
also induces the following exact sequence
\begin{equation}\label{SSdualE(-1)twist}
0\to \mathcal{S}^{\vee}\otimes\mathcal{E}(-1)
\to \mathcal{E}(-1)^{\oplus 4}\to \mathcal{S}^{\vee}\otimes\mathcal{E}
\to 0.
\end{equation}
This exact sequence \eqref{SSdualE(-1)twist} and 
an exact sequence
\begin{equation}\label{SdualE(-1)SdualEtoRestriction}
0\to \mathcal{S}^{\vee}\otimes \mathcal{E}(-1)
\to 
\mathcal{S}^{\vee}\otimes \mathcal{E}
\to 
\mathcal{S}^{\vee}\otimes \mathcal{E}|_{\mathbb{Q}^2}
\to 0
\end{equation}
will be used to compute $\Ext^q(\mathcal{S},\mathcal{E})$.

\section{The case where $\mathcal{E}|_{\mathbb{Q}^2}$ belongs to Case
(1) of Theorem~\ref{Chern(2,2)}}\label{Case(1)OfTheoremChern(2,2)}

The assumption~\eqref{h^0(E(-2))vanish} implies that this case does not arise.
Indeed, if  $\mathcal{E}|_{\mathbb{Q}^2}\cong \mathcal{O}(2,2)
\oplus \mathcal{O}^{\oplus r-1}$,
then $h^q(\mathcal{E}(-1)|_{\mathbb{Q}^2})=0$ for $q>0$.
Moreover $c_2h=0$. Hence 
$h^q(\mathcal{E}(-1))=0$ for $q>0$
by \eqref{c2h<4vanishing}.
This implies that $h^q(\mathcal{E}(-2))=0$ for $q\geq 2$.
The assumption~\eqref{h^0(E(-2))vanish} then shows that 
\[0\geq -h^1(\mathcal{E}(-2))=\chi(\mathcal{E}(-2))=
1+\dfrac{1}{2}c_3\]
by \eqref{e(-2)RR}. This contradicts \eqref{c3nonnegative}.
Hence this case does not arise.

\section{The case where $\mathcal{E}|_{\mathbb{Q}^2}$ belongs to Case
(2) of Theorem~\ref{Chern(2,2)}}

Suppose that 
\[\mathcal{E}|_{\mathbb{Q}^2}\cong \mathcal{O}(2,1)\oplus\mathcal{O}(0,1)
\oplus \mathcal{O}^{\oplus r-2}.\]
Then $h^0(\mathcal{E}(-1)|_{\mathbb{Q}^2})=2$
and $h^q(\mathcal{E}(-1)|_{\mathbb{Q}^2})=0$ for $q>0$.
Moreover $c_2h=2$. Hence 
\[h^q(\mathcal{E}(-1))=0\textrm{ for }q>0\]
by \eqref{c2h<4vanishing}.
It then follows from \eqref{e(-1)RR} 
and \eqref{c3nonnegative}
that 
$h^0(\mathcal{E}(-1))=\chi(\mathcal{E}(-1))=2+\dfrac{1}{2}c_3\geq 2$.
On the other hand, we have
$h^0(\mathcal{E}(-1))\leq h^0(\mathcal{E}(-1)|_{\mathbb{Q}^2})=2$
by \eqref{h^0(E(-2))vanish}.
Therefore 
the restriction map $H^0(\mathcal{E}(-1))\to 
H^0(\mathcal{E}(-1)|_{\mathbb{Q}^2})$
is an isomorphism,
\[h^0(\mathcal{E}(-1))=2,\textrm{ and }c_3=0.\]
Hence we see that 
\[h^q(\mathcal{E}(-2))=0 \textrm{ for all } q.\]
Since 
$\mathcal{E}(-2)|_{\mathbb{Q}^2}\cong \mathcal{O}(0,-1)\oplus\mathcal{O}(-2,-1)
\oplus \mathcal{O}(-2,-2)^{\oplus r-2}$,
we have $h^q(\mathcal{E}(-2)|_{\mathbb{Q}^2})=0$ for $q<2$
and $h^2(\mathcal{E}(-2)|_{\mathbb{Q}^2})=r-2$.
Therefore 
\[h^q(\mathcal{E}(-3))=0 \textrm{ for } q<3,
\textrm{ and }h^3(\mathcal{E}(-3))=r-2.\]
Next we will compute $\Ext^q(\mathcal{S},\mathcal{E}(-1))$.
Since 
\[\mathcal{S}^{\vee}\otimes\mathcal{E}(t)|_{\mathbb{Q}^2}
\cong (\mathcal{O}(-1,0)\oplus\mathcal{O}(0,-1))
\otimes
(\mathcal{O}(2+t,1+t)\oplus\mathcal{O}(t,1+t)\oplus\mathcal{O}(t,t)^{\oplus r-2}),
\]
we see that 
$h^q(\mathcal{S}^{\vee}\otimes\mathcal{E}(t)|_{\mathbb{Q}^2})=0$
for $q>0$ and $t\geq 0$.
Hence it follows from \eqref{ExtSE(1)q2ijouvanishing}
that 
\[\Ext^q(\mathcal{S},\mathcal{E}(-1))=0 \textrm{ for } q\geq 2.\]
Since $c_2h=2$ and $c_3=0$,
the formula~\eqref{SERR(-1)} shows that
\[h^0(\mathcal{S}^{\vee}\otimes \mathcal{E}(-1))
=h^1(\mathcal{S}^{\vee}\otimes \mathcal{E}(-1)).\]
Set $a=h^0(\mathcal{S}^{\vee}\otimes \mathcal{E}(-1))$.
Note that 
$\mathcal{S}^{\vee}\otimes \mathcal{E}(-1)$ fits in an exact sequence
\[
0\to \mathcal{S}^{\vee}\otimes \mathcal{E}(-2)\to 
\mathcal{E}(-2)^{\oplus 4}\to \mathcal{S}^{\vee}\otimes\mathcal{E}(-1)\to 0.
\]
by \eqref{SSdual} and \eqref{canonicalIsom}.
Since $h^q(\mathcal{E}(-2))=0$ for all $q$,
this exact sequence shows that 
\[h^q(\mathcal{S}^{\vee}\otimes \mathcal{E}(-2))=
\begin{cases}
0& \textrm{if }q=0, 3\\
a& \textrm{otherwise.}
\end{cases}
\]
On the other hand, we have an exact sequence
\begin{equation}\label{toorisugari}
0\to \mathcal{S}^{\vee}\otimes \mathcal{E}(-2)\to 
\mathcal{S}^{\vee}\otimes\mathcal{E}(-1)
\to (\mathcal{S}^{\vee}\otimes\mathcal{E}(-1))|_{\mathbb{Q}^2}
\to 0.
\end{equation}
Since 
\[\mathcal{S}^{\vee}\otimes\mathcal{E}(-1)|_{\mathbb{Q}^2}
\cong (\mathcal{O}(-1,0)\oplus\mathcal{O}(0,-1))
\otimes
(\mathcal{O}(1,0)\oplus\mathcal{O}(-1,0)\oplus\mathcal{O}(-1,-1)^{\oplus r-2}),
\]
we see that 
\[h^q(\mathcal{S}^{\vee}\otimes \mathcal{E}(-1)|_{\mathbb{Q}^2})=
\begin{cases}
1& \textrm{if }q=0, 1\\
0& \textrm{if }q=2, 3.
\end{cases}
\]
Hence the exact sequence \eqref{toorisugari} implies that $a=1$.

We apply 
to $\mathcal{E}(-1)$
the Bondal spectral sequence~\eqref{BondalSpectral}.
We have 
$\Ext^3(G,\mathcal{E}(-1))
\cong 
S_3^{\oplus r-2}$,
$\Ext^2(G,\mathcal{E}(-1))
=0$,
and 
$\Ext^1(G,\mathcal{E}(-1))
\cong
S_1$.
Moreover $\Hom(G,\mathcal{E}(-1))$ fits in an exact sequence
\[
0\to S_0^{\oplus 2}\to \Hom(G,\mathcal{E}(-1))\to S_1\to 0.
\]
Now Lemma~\ref{S2Arilemma} shows that 
$E_2^{p,3}=0$ unless $p=-3$, that $E_2^{-3,3}\cong \mathcal{O}(-1)^{\oplus r-2}$,
that $E_2^{p,2}=0$ for all $p$, that $E_2^{p,1}=0$ unless $p=-1$,
that $E_2^{-1,1}\cong \mathcal{S}(-1)$,
and that 
a distinguished triangle
\[\mathcal{O}^{\oplus 2}\to \Hom(G,\mathcal{E}(-1))\lotimes_A G
\to \mathcal{S}(-1)[1]\to \]
exists.
Hence we have the following exact sequence:
\begin{equation}\label{tochuu}
0\to E_2^{-1,0}\to \mathcal{S}(-1)\to \mathcal{O}^{\oplus 2}\to E_2^{0,0}\to 0.
\end{equation}
Note here that $E_2^{-1,0}\cong E_{\infty}^{-1,0}=0$.
Hence we see that $E_2^{0,0}$ is a non-zero torsion sheaf.
On the other hand, 
$\mathcal{E}(-1)$ has $E_2^{0,0}$ as a subsheaf,
so that $E_2^{0,0}$ must be torsion-free.
This is a contradiction. 
Therefore this case does not arise.

\section{The case where $\mathcal{E}|_{\mathbb{Q}^2}$ belongs to Case
(3) of Theorem~\ref{Chern(2,2)}}

Suppose that 
$\mathcal{E}|_{\mathbb{Q}^2}\cong 
\mathcal{O}(1,1)^{\oplus 2}
\oplus \mathcal{O}^{\oplus r-2}$.
Then $c_2.h=2$.
Hence $h^q(\mathcal{E}(-1))=0$ for $q>0$ by \eqref{c2h<4vanishing}.
Since $h^q(\mathcal{E}(-1)|_{\mathbb{Q}^2})=0$ for $q>0$,
this implies that $h^q(\mathcal{E}(-2))=0$ for $q\geq 2$.
The assumption~\eqref{h^0(E(-2))vanish} together with \eqref{e(-2)RR} and \eqref{c3nonnegative}
shows that 
\[0\geq -h^1(\mathcal{E}(-2))=\chi(\mathcal{E}(-2))
=\frac{1}{2}c_3\geq 0.\]
Hence $h^1(\mathcal{E}(-2))=0$ and $c_3=0$.
Thus $h^0(\mathcal{E}(-1))=h^0(\mathcal{E}(-1)|_{\mathbb{Q}^2})=2$.
Since $h^q(\mathcal{E}(-2))=0$ for any $q$, we see that 
$h^q(\mathcal{E}(-3))=h^{q-1}(\mathcal{E}(-2)|_{\mathbb{Q}^2})$
for all $q$.
Hence $h^q(\mathcal{E}(-3))=0$ unless $q=3$
and $h^3(\mathcal{E}(-3))=r-2$.
Since 
\[\mathcal{S}^{\vee}\otimes\mathcal{E}(t)|_{\mathbb{Q}^2}
\cong (\mathcal{O}(-1,0)\oplus\mathcal{O}(0,-1))
\otimes
(\mathcal{O}(1+t,1+t)^{\oplus 2}\oplus\mathcal{O}(t,t)^{\oplus r-2}),
\]
we see that 
$h^q(\mathcal{S}^{\vee}\otimes\mathcal{E}(t)|_{\mathbb{Q}^2})=0$
for $q>0$ and $t\geq -1$.
Hence it follows from \eqref{ExtSE(1)q2ijouvanishing}
that 
$\Ext^q(\mathcal{S},\mathcal{E}(-t))=0$ for  $q\geq 2$ and $t=0,1,2$.
Since the exact sequence \eqref{SSdual} together with \eqref{canonicalIsom} induces
an exact sequence
\[
0\to \mathcal{S}^{\vee}\otimes \mathcal{E}(-2)\to 
\mathcal{E}(-2)^{\oplus 4}\to \mathcal{S}^{\vee}\otimes\mathcal{E}(-1)\to 0,
\]
the vanishing $h^1(\mathcal{E}(-2))=0$ implies 
that $h^1(\mathcal{S}^{\vee}\otimes\mathcal{E}(-1))=0$.
Since $h^0(\mathcal{S}^{\vee}\otimes \mathcal{E}(-1)|_{\mathbb{Q}^2})=0$,
this implies that $h^1(\mathcal{S}^{\vee}\otimes \mathcal{E}(-2))=0$.
Hence $h^0(\mathcal{S}^{\vee}\otimes\mathcal{E}(-1))=
h^0(\mathcal{S}^{\vee}\otimes \mathcal{E}(-1)|_{\mathbb{Q}^2})=0$.
We apply to $\mathcal{E}(-1)$ 
the Bondal spectral sequence~\eqref{BondalSpectral}.
We see that $\Hom(G,\mathcal{E}(-1))\cong S_0^{\oplus 2}$,
that $\Ext^q(G,\mathcal{E}(-1))=0$ for $q=1,2$,
and that 
$\Ext^3(G,\mathcal{E}(-1))
\cong 
S_3^{\oplus r-2}$.
Hence $E_2^{p,q}=0$ unless $q=0$ or $q=3$,
$E_2^{p,0}=0$ unless $p=0$, $E_2^{0,0}=\mathcal{O}^{\oplus 2}$,
$E_2^{p,3}=0$ unless $p=-3$, and $E_2^{-3,3}=\mathcal{O}(-1)^{\oplus r-2}$
by Lemma~\ref{S2Arilemma}.
Therefore $\mathcal{E}(-1)$ fits in an exact sequence
\[
0\to \mathcal{O}^{\oplus 2}\to \mathcal{E}(-1)\to \mathcal{O}(-1)^{\oplus r-2}
\to 0.
\]
Hence $\mathcal{E}\cong \mathcal{O}(1)^{\oplus 2}\oplus\mathcal{O}^{\oplus r-2}$.
This is Case (2) of Theorem~\ref{Chern2}.

\section{The case where $\mathcal{E}|_{\mathbb{Q}^2}$ belongs to Case
(4) of Theorem~\ref{Chern(2,2)}}

Suppose that $\mathcal{E}|_{\mathbb{Q}^2}$ fits in an exact sequence
\[
0\to \mathcal{O}
\to 
\mathcal{O}(1,1)\oplus
\mathcal{O}(1,0)\oplus  
\mathcal{O}(0,1)\oplus  
\mathcal{O}^{\oplus r-2}\to \mathcal{E}|_{\mathbb{Q}^2}\to 0.
\]
Then $c_2h=3$.
Hence $h^q(\mathcal{E}(-1))=0$ for $q>0$ by \eqref{c2h<4vanishing}.
Note that $h^q(\mathcal{E}(-1)|_{\mathbb{Q}^2})=0$ for $q>0$
and that $h^0(\mathcal{E}(-1)|_{\mathbb{Q}^2})=1$.
Hence $h^q(\mathcal{E}(-2))=0$ for $q\geq 2$.
The assumption~\eqref{h^0(E(-2))vanish} together with \eqref{e(-2)RR} and \eqref{c3nonnegative}
shows that 
\[0\geq -h^1(\mathcal{E}(-2))=\chi(\mathcal{E}(-2))
=-\frac{1}{2}+\frac{1}{2}c_3\geq -\frac{1}{2}.\]
Hence $h^1(\mathcal{E}(-2))=0$ and $c_3=1$.
Now that $h^q(\mathcal{E}(-2))=0$
for any $q$, 
we have $h^q(\mathcal{E}(-3))=h^{q-1}(\mathcal{E}(-2)|_{\mathbb{Q}^2})$
for any $q$.
Set 
$a=h^1(\mathcal{E}(-2)|_{\mathbb{Q}^2})$.
Then $a=0$ or $1$,
and $h^2(\mathcal{E}(-2)|_{\mathbb{Q}^2})=r-3+a$.
Hence we see that $h^q(\mathcal{E}(-3))=0$ for $q\leq 1$,
that $h^2(\mathcal{E}(-3))=a$, and that $h^3(\mathcal{E}(-3))=r-3+a$.
Moreover the assumption~\eqref{h^0(E(-2))vanish} implies that 
$h^0(\mathcal{E}(-1))=h^0(\mathcal{E}(-1)|_{\mathbb{Q}^2})=1$.
Since $\mathcal{E}|_{\mathbb{Q}^2}(-2,-1)$ fits in an exact sequence
\[
\begin{split}
0\to \mathcal{O}(-2,-1)
\to 
\mathcal{O}(-1,0)\oplus
\mathcal{O}(-1,-1)\oplus  
\mathcal{O}(-2,0)\oplus  
&\mathcal{O}(-2,-1)^{\oplus r-2}\\
&\to \mathcal{E}|_{\mathbb{Q}^2}(-2,-1)\to 0,
\end{split}
\]
we see that $h^q(\mathcal{E}|_{\mathbb{Q}^2}(-2,-1))=0$ unless $q=1$.
Hence $h^q(\mathcal{S}^{\vee}\otimes\mathcal{E}(-1)|_{\mathbb{Q}^2})=0$
unless $q=1$.
Note that $h^q(\mathcal{S}^{\vee}\otimes\mathcal{E}(t)|_{\mathbb{Q}^2})=0$
for $t\geq 0$ and $q\geq 1$.
Hence it follows from \eqref{ExtSE(1)q2ijouvanishing}
that 
$\Ext^q(\mathcal{S},\mathcal{E}(-t))=0$ for  $q\geq 2$ and $t=0,1$.
Note that $\mathcal{S}^{\vee}\otimes \mathcal{E}(-2)$
is a subbundle of $\mathcal{E}(-2)^{\oplus 4}$ by \eqref{SSdual}.
Since $h^0(\mathcal{E}(-2))=0$, this implies that  
$h^0(\mathcal{S}^{\vee}\otimes\mathcal{E}(-2))=0$.
Since we have 
an exact sequence
\[
0\to \mathcal{S}^{\vee}\otimes \mathcal{E}(-2)
\to \mathcal{S}^{\vee}\otimes\mathcal{E}(-1)
\to \mathcal{S}^{\vee}\otimes\mathcal{E}(-1)|_{\mathbb{Q}^2}\to 0
\]
and $h^0(\mathcal{S}^{\vee}\otimes\mathcal{E}(-1)|_{\mathbb{Q}^2})=0$,
we infer that $h^0(\mathcal{S}^{\vee}\otimes\mathcal{E}(-1))=0$.
Now, from \eqref{SERR(-1)},
it follows that 
\[
-h^1(\mathcal{S}^{\vee}\otimes\mathcal{E}(-1))
=\chi(\mathcal{S}^{\vee}\otimes\mathcal{E}(-1))
=4-2\cdot 3+1=-1.
\]
We apply to $\mathcal{E}(-1)$ 
the Bondal spectral sequence~\eqref{BondalSpectral}.
We 
have
the following isomorphisms:
$\Ext^3(G,\mathcal{E}(-1))\cong S_3^{\oplus r-3+a}$;
$\Ext^2(G,\mathcal{E}(-1))\cong S_3^{\oplus a}$;
$\Ext^1(G,\mathcal{E}(-1))\cong S_1$;
$\Hom(G,\mathcal{E}(-1))\cong S_0$.
Lemma~\ref{S2Arilemma} then shows 
that $E_2^{p,q}=0$ unless $(p,q)=(-3,3)$, $(-3,2)$, $(-1,1)$, or $(0,0)$,
that $E_2^{-3,3}=\mathcal{O}(-1)^{\oplus r-3+a}$,
that $E_2^{-3,2}=\mathcal{O}(-1)^{\oplus a}$, 
that $E_2^{-1,1}=\mathcal{S}(-1)$,
and 
that $E_2^{0,0}=\mathcal{O}$.
Hence $E_3^{-3,2}=0$ and $E_3^{-1,1}$ fits in the following exact sequence:
\[0\to \mathcal{O}(-1)^{\oplus a}\to \mathcal{S}(-1)\to E_3^{-1,1}\to 0.\]
Moreover $\mathcal{E}(-1)$ has a filtration 
$\mathcal{O}\subset F(\mathcal{E}(-1))\subset \mathcal{E}(-1)$
such that $F(\mathcal{E}(-1))$ fits in the following exact sequences:
\[0\to F(\mathcal{E}(-1))\to \mathcal{E}(-1)\to \mathcal{O}(-1)^{\oplus r-3}\to 0;\]
\[0\to \mathcal{O}\to F(\mathcal{E}(-1))\to E_3^{-1,1}\to 0.\]
In particular, we see that $F(\mathcal{E}(-1))$ is a vector bundle, since so is $\mathcal{E}(-1)$.
On the other hand, since $\Ext^1(\mathcal{S}(-1),\mathcal{O})=0$,
$F(\mathcal{E}(-1))$ fits in the following exact sequence:
\[
0\to \mathcal{O}(-1)^{\oplus a}\to \mathcal{O}\oplus\mathcal{S}(-1)\to F(\mathcal{E}(-1))\to 0.
\]
This implies that $a=0$. Indeed, if $a=1$, then $F(\mathcal{E}(-1))$ cannot be a vector bundle,
since the intersection of a line
and a hyperplane section 
cannot be empty.
Therefore $F(\mathcal{E}(-1))\cong  \mathcal{O}\oplus \mathcal{S}(-1)$,
and thus $\mathcal{E}\cong \mathcal{O}(1)\oplus \mathcal{S}\oplus \mathcal{O}^{\oplus r-3}$.
This is Case (3) of Theorem~\ref{Chern2}.

\section{The case where $\mathcal{E}|_{\mathbb{Q}^2}$ belongs to Case
(5) of Theorem~\ref{Chern(2,2)}}

Suppose that $\mathcal{E}|_{\mathbb{Q}^2}$ fits in an exact sequence
\[
0\to \mathcal{O}(-1,-1)
\to 
\mathcal{O}(1,1)\oplus 
\mathcal{O}^{\oplus r}\to \mathcal{E}|_{\mathbb{Q}^2}\to 0.
\]
Then $c_2h=4$.
Note that 
\begin{equation}\label{SdualErestrictionInCase6}
h^q(\mathcal{S}^{\vee}\otimes \mathcal{E}|_{\mathbb{Q}^2})
=
\begin{cases}
4& \textrm{if}\quad  q= 0\\
0& \textrm{if}\quad  q\neq 0,
\end{cases}
\end{equation}
and that
\begin{equation}\label{CohomologyOfE(-1)RestInCase6InTh2-1}
h^q(\mathcal{E}(-1)|_{\mathbb{Q}^2})
=
\begin{cases}
1& \textrm{if}\quad  q= 0, 1\\
0& \textrm{if}\quad  q\neq 0, 1.
\end{cases}
\end{equation}
Hence we have 
\[h^0(\mathcal{E}(-1))\leq 1\]
by \eqref{h^0(E(-2))vanish}.

\subsection{Suppose that $h^0(\mathcal{E}(-1))=1$.}\label{Case6h^0(E(-1))=1} 
Lemma~\ref{c3lowerbound} then shows that $c_3\geq 4$.
Hence $H^q(\mathcal{E}(-1))$ vanishes for $q>0$ by \eqref{higherE(-1)vanish}.
The formula~\eqref{e(-1)RR} then shows that 
\[h^0(\mathcal{E}(-1))=-1+\dfrac{1}{2}c_3.\]
Thus we have $c_3=4$.
We also see that $h^q(\mathcal{E}(-2))=0$ unless $q=2$ and that $h^2(\mathcal{E}(-2))=1$
by 
\eqref{CohomologyOfE(-1)RestInCase6InTh2-1}
and \eqref{h^0(E(-2))vanish}.
We have $h^0(\mathcal{E})=r+5$.
Since we have an exact sequence 
\[
0\to \mathcal{S}^{\vee}\otimes \mathcal{E}(-2)
\to 
\mathcal{E}(-2)^{\oplus 4}
\to 
\mathcal{S}^{\vee}\otimes \mathcal{E}(-1)
\to 0,
\]
we see that $h^0(\mathcal{S}^{\vee}\otimes \mathcal{E}(-2))=0$
and that $h^3(\mathcal{S}^{\vee}\otimes \mathcal{E}(-1))=0$.
Note that 
$h^0(\mathcal{S}^{\vee}\otimes \mathcal{E}(-1)|_{\mathbb{Q}^2})=0$.
Since we have an exact sequence
\[
0\to \mathcal{S}^{\vee}\otimes \mathcal{E}(-2)
\to 
\mathcal{S}^{\vee}\otimes \mathcal{E}(-1)
\to 
\mathcal{S}^{\vee}\otimes \mathcal{E}(-1)|_{\mathbb{Q}^2}
\to 0,
\]
we infer that $h^0(\mathcal{S}^{\vee}\otimes \mathcal{E}(-1))=0$.
Since we have an exact sequence \eqref{SSdualE(-1)twist},
we see that $h^q(\mathcal{S}^{\vee}\otimes \mathcal{E})=0$ for $q\geq 2$.
The exact sequence~\eqref{SdualE(-1)SdualEtoRestriction}
together with \eqref{SdualErestrictionInCase6}
shows that $h^2(\mathcal{S}^{\vee}\otimes \mathcal{E}(-1))=0$.
Now the formula~\eqref{SERR(-1)}
shows that
\[-h^1(\mathcal{S}^{\vee}\otimes\mathcal{E}(-1))
=\chi(\mathcal{S}^{\vee}\otimes \mathcal{E}(-1))=0,\]
since  $c_3=4$ and $c_2h=4$.
The exact sequence~\eqref{SSdualE(-1)twist}
then implies that 
$h^q(\mathcal{S}^{\vee}\otimes\mathcal{E})=0$ unless $q=0$
and that $h^0(\mathcal{S}^{\vee}\otimes\mathcal{E})=4$.
Since $h^0(\mathcal{E}(-1))=1$, we have an injection $\mathcal{O}(1)\to \mathcal{E}$.
Let $\mathcal{F}$ be its cokernel: we have the following exact sequence:
\[
0\to \mathcal{O}(1)\to \mathcal{E}\to \mathcal{F}\to 0.
\]
We 
apply to $\mathcal{F}$ 
the Bondal spectral sequence~\eqref{BondalSpectral}.
We see that $h^q(\mathcal{F})=0$ unless $q=0$
and that $h^0(\mathcal{F})=r$.
Moreover $h^q(\mathcal{F}(-1))=0$ for any $q$,
$h^q(\mathcal{F}(-2))=0$ unless $q=2$,
and $h^2(\mathcal{F}(-2))=1$. 
Finally we have $h^q(\mathcal{S}^{\vee}\otimes \mathcal{F})=0$ for all $q$.
Therefore $\Ext^q(G,\mathcal{F})=0$ for $q=3$ and  $1$,
$\Ext^2(G,\mathcal{F})\cong S_3$,
and $\Hom(G,\mathcal{F})\cong S_0^{\oplus r}$.
Hence $E_2^{p,q}=0$ unless $(p.q)=(-3,2)$ or $(0,0)$,
$E_2^{-3,2}=\mathcal{O}(-1)$, and $E_2^{0,0}=\mathcal{O}^{\oplus r}$
by Lemma~\ref{S2Arilemma}.
Thus we have an exact sequence
\[
0\to \mathcal{O}(-1)\to \mathcal{O}^{\oplus r}\to \mathcal{F}\to 0.
\]
Therefore $\mathcal{E}$ belongs to Case (4) of Theorem~\ref{Chern2}.

\subsection{Suppose that $h^0(\mathcal{E}(-1))=0$.} 
Then $h^0(\mathcal{S}^{\vee}\otimes \mathcal{E}(-1))=0$ 
by  \eqref{SSdualE(-1)twist}.
Note that $H^q(\mathcal{E}|_{\mathbb{Q}^2})$ vanishes for all $q>0$.
Since $h^q(\mathcal{E})=0$ for all $q>0$ by \eqref{firstvanishing},
we have $h^q(\mathcal{E}(-1))=0$ for all $q\geq 2$.
Hence \eqref{e(-1)RR} 
and \eqref{c3nonnegative}
imply that 
\[0\geq -h^1(\mathcal{E}(-1))=\chi(\mathcal{E}(-1))=-1+\dfrac{1}{2}c_3\geq -1.\]
Therefore $(h^1(\mathcal{E}(-1)),c_3)$ is either $(0,2)$
or $(1,0)$.
Since $h^3(\mathcal{E}(-1))=0$,
we first have $h^3(\mathcal{S}^{\vee}\otimes \mathcal{E})=0$ 
by \eqref{SSdualE(-1)twist}.
Secondly we have $h^3(\mathcal{S}^{\vee}\otimes \mathcal{E}(-1))=0$
by \eqref{SdualErestrictionInCase6} and \eqref{SdualE(-1)SdualEtoRestriction}.
Thirdly we have $h^2(\mathcal{S}^{\vee}\otimes \mathcal{E})=0$ 
by \eqref{SSdualE(-1)twist} since $h^2(\mathcal{E}(-1))=0$.
Finally we have $h^2(\mathcal{S}^{\vee}\otimes \mathcal{E}(-1))=0$ 
by \eqref{SdualErestrictionInCase6} and \eqref{SdualE(-1)SdualEtoRestriction}.
Hence 
\begin{equation}\label{(3-1)h1SdualE(-1)noc3niyorubaaiwakeInCase6OfThm2-1}
-h^1(\mathcal{S}^{\vee}\otimes\mathcal{E}(-1))
=\chi(\mathcal{S}^{\vee}\otimes\mathcal{E}(-1))
=-4+c_3
\end{equation}
by \eqref{SERR(-1)}.
We apply to $\mathcal{E}$ the Bondal spectral sequence~\eqref{BondalSpectral}.

\subsubsection{Suppose that $(h^1(\mathcal{E}(-1)),c_3)=(0,2)$.}
Then  $h^1(\mathcal{S}^{\vee}\otimes \mathcal{E})=0$ 
by \eqref{SSdualE(-1)twist}.
Moreover $h^1(\mathcal{S}^{\vee}\otimes\mathcal{E}(-1))=2$
by \eqref{(3-1)h1SdualE(-1)noc3niyorubaaiwakeInCase6OfThm2-1}.
Hence we have $h^0(\mathcal{S}^{\vee}\otimes \mathcal{E})=2$ 
by \eqref{SSdualE(-1)twist}.
Since $h^q(\mathcal{E}(-1)|_{\mathbb{Q}^2})=1$ for $q=0$, $1$
and $h^q(\mathcal{E}(-1)|_{\mathbb{Q}^2})=0$ for $q=2$, $3$,
we 
infer 
that $h^q(\mathcal{E}(-2))=1$ for $q=1$, $2$, 
and that 
$h^q(\mathcal{E}(-2))=0$ 
unless $q=1$ or $2$.
Since $h^0(\mathcal{E}|_{\mathbb{Q}^2})=r+4$, we 
see that 
$h^0(\mathcal{E})=r+4$.
Therefore we have 
an exact sequence
\[
0\to S_0^{\oplus r+4}\to \Hom(G,\mathcal{E})\to S_1^{\oplus 2}\to 0
\]
and the following:
$\Ext^1(G,\mathcal{E})\cong S_3$;
$\Ext^2(G,\mathcal{E})\cong S_3$; and 
$\Ext^3(G,\mathcal{E})=0$.
Therefore Lemma~\ref{S2Arilemma} implies 
that $E_2^{p,q}=0$ unless $(p,q)=(-3,1)$, $(-3,2)$, $(-1,0)$, or $(0,0)$,
that $E_2^{-3,1}\cong \mathcal{O}(-1)$, 
that $E_2^{-3,2}\cong \mathcal{O}(-1)$,
and that there is  
an exact sequence
\[
0\to E_2^{-1,0}\to \mathcal{S}(-1)^{\oplus 2}\to \mathcal{O}^{\oplus r+4}
\to E_2^{0,0}\to 0.\]
It follows from  the Bondal spectral sequence~\eqref{BondalSpectral}
that  $E_2^{-3,1}\cong E_2^{-1,0}$,
that $E_2^{-3,2}\cong E_3^{-3,2}$, that $E_2^{0,0}\cong E_3^{0,0}$,
and that there is  
an exact sequence
\[0\to E_3^{-3,2}\to E_3^{0,0}\to \mathcal{E}\to 0.\]
Hence we obtain the following exact sequences:
\[
0\to \mathcal{O}(-1)\to \mathcal{S}(-1)^{\oplus 2}\to \mathcal{O}^{\oplus r+4}
\to E_3^{0,0}\to 0;\]
\[0\to \mathcal{O}(-1)\to E_3^{0,0}\to \mathcal{E}\to 0. 
\]
The latter exact sequence 
shows 
that 
$E_3^{0,0}$ is a vector bundle since so is $\mathcal{E}$.
The former exact sequence then splits into the following two exact sequences
with $\mathcal{G}$ a vector bundle of rank three:
\[0\to \mathcal{O}(-1)\to \mathcal{S}(-1)^{\oplus 2}\to \mathcal{G}\to 0;\]
\[0\to \mathcal{G}\to \mathcal{O}^{\oplus r+4}\to E_3^{0,0}\to 0.\]
The latter exact sequence shows 
that the dual $\mathcal{G}^{\vee}$ of $\mathcal{G}$ is globally generated.
The injection $\mathcal{O}(-1)\to \mathcal{S}(-1)^{\oplus 2}$ 
in the former exact sequence 
gives rise to 
two global sections $s_0$, $s_1$ of $\mathcal{S}$,
and we infer that $(s_0)_0\cap (s_1)_0=\emptyset$ since $\mathcal{G}$ is a vector bundle.
Hence $s_0$ and $s_1$ are linearly independent.
We also see that $\mathcal{G}^{\vee}$
fits in the following exact sequence:
\[0\to \mathcal{G}^{\vee}\to \mathcal{S}^{\oplus 2}\to \mathcal{O}(1)\to 0.\]
Note that the induced map $H^0(\mathcal{S})^{\oplus 2}\to H^0(\mathcal{O}(1))$ 
sends $(t_0,t_1)$ to $s_0\wedge t_0+s_1\wedge t_1$,
and 
Lemma~\ref{wedgeproductOfS}
implies that it is surjective.
Therefore $h^0(\mathcal{G}^{\vee})=3$.
Since $\mathcal{G}^{\vee}$ is a globally generated vector bundle of rank three, 
this  implies that $\mathcal{G}^{\vee}\cong \mathcal{O}^{\oplus 3}$.
On the other hand, the exact sequence above shows that $c_1(\mathcal{G}^{\vee})=1$.
This is a contradiction. Hence 
the case $(h^1(\mathcal{E}(-1)),c_3)=(0,2)$
does not arise.

\subsubsection{Suppose that $(h^1(\mathcal{E}(-1)),c_3)=(1,0)$.}
Then $h^1(\mathcal{S}^{\vee}\otimes\mathcal{E}(-1))=4$
by \eqref{(3-1)h1SdualE(-1)noc3niyorubaaiwakeInCase6OfThm2-1}.
Set $a:=h^0(\mathcal{S}^{\vee}\otimes \mathcal{E})$.
Then $h^1(\mathcal{S}^{\vee}\otimes \mathcal{E})=a$ by 
\eqref{SSdualE(-1)twist}.
From \eqref{CohomologyOfE(-1)RestInCase6InTh2-1},
it follows 
that $h^q(\mathcal{E}(-2))=0$ unless $q=1$ or $2$
and that $(h^1(\mathcal{E}(-2)),h^2(\mathcal{E}(-2)))=(1,0)$ or $(2,1)$.
Note also that $h^0(\mathcal{E})=r+3$.

\paragraph{Suppose that $(h^1(\mathcal{E}(-2)),h^2(\mathcal{E}(-2)))=(1,0)$}
\label{NewCase(5)OfTh2.2h^0(E(-1))=0c3=0}
Then we see that $\Ext^3(G,\mathcal{E})=0$, that $\Ext^2(G,\mathcal{E})=0$,
that $\Ext^1(G,\mathcal{E})$ has a filtration $S_1^{\oplus a}\subset F\subset \Ext^1(G,\mathcal{E})$
of right $A$-modules such that the following sequences are exact:
\[
0\to F\to \Ext^1(G,\mathcal{E})\to S_3\to 0;\]
\[
0\to S_1^{\oplus a}\to F\to S_2\to 0,\]
and that $\Hom(G,\mathcal{E})$ fits in the following exact sequence 
of right $A$-modules:
\[
0\to S_0^{\oplus r+3}\to \Hom(G,\mathcal{E})\to S_1^{\oplus a}\to 0.
\]
These exact sequences induce the following distinguished triangles by Lemma~\ref{S2Arilemma}:
\[F\lotimes_AG\to \Ext^1(G,\mathcal{E})\lotimes_AG\to \mathcal{O}(-1)[3]\to;\]
\[\mathcal{S}(-1)[1]^{\oplus a}\to F\lotimes_AG\to T_{\mathbb{P}^4}(-2)|_{\mathbb{Q}^3}[2]\to;\]
\[\mathcal{O}^{\oplus r+3}\to \Hom(G,\mathcal{E})\lotimes_AG\to \mathcal{S}(-1)[1]^{\oplus a}\to.\]
By taking cohomologies, we obtain the following exact sequences
by \eqref{canonicalIsom}:
\[0\to E_2^{-3,1}\to \mathcal{O}(-1)\to \mathcal{H}^{-2}(F\lotimes_AG)
\to E_2^{-2,1}\to 0;\]
\begin{equation}\label{E_2^-1,1resolutionInCaseh1E(-2)=1,c_3=0}
0\to \mathcal{H}^{-2}(F\lotimes_AG)\to 
T_{\mathbb{P}^4}(-2)|_{\mathbb{Q}^3}
\xrightarrow{\psi_a}
\mathcal{S}^{\vee\oplus a}\to E_2^{-1,1}\to 0;
\end{equation}
\[0\to E_2^{-1,0}\to \mathcal{S}^{\vee\oplus a}
\to\mathcal{O}^{\oplus r+3}\to E_2^{0,0}\to 0.\]
Moreover we have the following exact sequences:
\[0\to E_2^{-2,1}\to E_2^{0,0}\to E_3^{0,0}\to 0;\]
\[0\to E_2^{-3,1}\to E_2^{-1,0}\to 0;\]
\[0\to E_3^{0,0}\to \mathcal{E}\to E_2^{-1,1}\to 0.\]
Since $\mathcal{E}$ is nef, $E_2^{-1,1}$ cannot admit negative degree quotients.
Hence it follows from Lemma~\ref{SupposedVeryImportant} that $a=0$.
Then $E_2^{-1,1}=0$, $E_2^{-3,1}=E_2^{-1,0}=0$, $E_2^{0,0}=\mathcal{O}^{\oplus r+3}$, 
and we have the following exact sequence:
\[0\to \mathcal{O}(-1)\to T_{\mathbb{P}^4}(-2)|_{\mathbb{Q}^3}\to E_2^{-2,1}\to 0.\]
Hence $\mathcal{E}$ fits in the following exact sequence:
\begin{equation}\label{procase(14)}
0\to \mathcal{O}(-1)\to T_{\mathbb{P}^4}(-2)|_{\mathbb{Q}^3}
\to \mathcal{O}^{\oplus r+3}\to \mathcal{E}\to 0.
\end{equation}
Since $T_{\mathbb{P}^4}(-2)|_{\mathbb{Q}^3}$ fits in 
an exact sequence
\[
0\to \mathcal{O}(-2)\to \mathcal{O}(-1)^{\oplus 5}\to T_{\mathbb{P}^4}(-2)|_{\mathbb{Q}^3}\to 0,
\]
the exact sequence~\eqref{procase(14)} induces the following exact sequence:
\[
0\to \mathcal{O}(-2)\to \mathcal{O}(-1)^{\oplus 4}\to \mathcal{O}^{\oplus r+3}\to \mathcal{E}\to 0.
\]
This is Case (9) of Theorem~\ref{Chern2}.

\paragraph{Suppose that $(h^1(\mathcal{E}(-2)),h^2(\mathcal{E}(-2)))=(2,1)$}
Then we see that $\Ext^3(G,\mathcal{E})=0$, 
that $\Ext^2(G,\mathcal{E})\cong S_3$,
that $\Ext^1(G,\mathcal{E})$ has a filtration $S_1^{\oplus a}\subset F\subset \Ext^1(G,\mathcal{E})$
of right $A$-modules such that the following sequences are exact:
\[
0\to F\to \Ext^1(G,\mathcal{E})\to S_3^{\oplus 2}\to 0;\]
\[
0\to S_1^{\oplus a}\to F\to S_2\to 0,\]
and that $\Hom(G,\mathcal{E})$ fits in the following exact sequence 
of right $A$-modules:
\[
0\to S_0^{\oplus r+3}\to \Hom(G,\mathcal{E})\to S_1^{\oplus a}\to 0.
\]
Lemma~\ref{S2Arilemma} implies that 
$\Ext^2(G,\mathcal{E})\lotimes_AG\cong \mathcal{O}(-1)[3]$
and that the three exact sequences above 
induce the following distinguished triangles:
\[F\lotimes_AG\to \Ext^1(G,\mathcal{E})\lotimes_AG\to \mathcal{O}(-1)^{\oplus 2}[3]\to;\]
\[\mathcal{S}(-1)[1]^{\oplus a}\to F\lotimes_AG\to T_{\mathbb{P}^4}(-2)|_{\mathbb{Q}^3}[2]\to;\]
\[\mathcal{O}^{\oplus r+3}\to \Hom(G,\mathcal{E})\lotimes_AG\to \mathcal{S}(-1)[1]^{\oplus a}\to.\]
By taking cohomologies, we see that $E_2^{p,2}=0$ unless $p=-3$,
that $E_2^{-3,2}\cong \mathcal{O}(-1)$, and that we have  the following exact sequences
by \eqref{canonicalIsom}:
\begin{equation}\label{MostComplicated}
0\to E_2^{-3,1}\to \mathcal{O}(-1)^{\oplus 2}\to \mathcal{H}^{-2}(F\lotimes_AG)
\to E_2^{-2,1}\to 0;
\end{equation}
\begin{equation}\label{OneOfApplicationOfpsi}
0\to \mathcal{H}^{-2}(F\lotimes_AG)\to 
T_{\mathbb{P}^4}(-2)|_{\mathbb{Q}^3}
\xrightarrow{\psi_a}
\mathcal{S}^{\vee\oplus a}\to E_2^{-1,1}\to 0;
\end{equation}
\[0\to E_2^{-1,0}\to \mathcal{S}^{\vee\oplus a}
\to\mathcal{O}^{\oplus r+3}\to E_2^{0,0}\to 0.\]
Moreover we have the following exact sequences:
\begin{equation}\label{oneOfapplicationOfpi}
0\to E_3^{-3,2}\to E_2^{-3,2}\xrightarrow{\pi} E_2^{-1,1}\to E_3^{-1,1}\to 0;
\end{equation}
\[0\to E_2^{-2,1}\to E_2^{0,0}\to E_3^{0,0}\to 0;\]
\[0\to E_2^{-3,1}\to E_2^{-1,0}\to 0;\]
\[0\to E_3^{-3,2}\to E_3^{0,0}\to E_4^{0,0}\to 0;\]
\[0\to E_4^{0,0}\to \mathcal{E}\to E_3^{-1,1}\to 0.\]
Since $\mathcal{E}$ is nef, $E_3^{-1,1}$ cannot admit negative degree quotients.
If $a>0$, it follows from Lemmas~\ref{SupposedVeryImportantPart2}
and \ref{SupposedVeryImportant} 
that $a=1$, that $E_3^{-1,1}=0$,  
that $E_3^{-3,2}\cong \mathcal{I}_L(-1)$ for some line 
$L\subset\mathbb{Q}^3$,
that $E_2^{-1,1}\cong \mathcal{O}_L(-1)$,
and that $ \mathcal{H}^{-2}(F\lotimes_AG)\cong \mathcal{O}(-1)^{\oplus 2}$.
Therefore $\mathcal{E}\cong E_4^{0,0}$ and 
the exact sequence \eqref{MostComplicated} becomes 
the following exact sequence:
\[0\to E_2^{-3,1}\to \mathcal{O}(-1)^{\oplus 2}\to \mathcal{O}(-1)^{\oplus 2}
\to E_2^{-2,1}\to 0.\]
Set $\mathcal{O}(-1)^{\oplus b}\cong E_2^{-3,1}$ for some non-negative integer $b\leq 2$.
Then $E_2^{-2,1}\cong \mathcal{O}(-1)^{\oplus b}$ and we have the following exact sequences:
\[0\to \mathcal{O}(-1)^{\oplus b}\to \mathcal{S}^{\vee}
\to\mathcal{O}^{\oplus r+3}\to E_2^{0,0}\to 0;\]
\[0\to\mathcal{O}(-1)^{\oplus b} \to E_2^{0,0}\to E_3^{0,0}\to 0;\]
\[0\to \mathcal{I}_L(-1)\to E_3^{0,0}\to \mathcal{E}\to 0.\]
Since $\mathcal{O}^{\oplus r+3}$ is torsion-free
and $\mathcal{S}^{\vee}$ is not isomorphic to $\mathcal{O}^{\oplus 2}$, we see that $b\leq 1$.
Note here that $E_3^{0,0}$ is torsion-free, and so is $E_2^{0,0}$.
If $b=1$, we get 
an exact sequence
\[0\to \mathcal{I}_M\to \mathcal{O}^{\oplus r+3}\to E_2^{0,0}\to 0\]
for some line $M$
in $\mathbb{Q}^3$.
Since we can extend $\mathcal{I}_M\to \mathcal{O}^{\oplus r+3}$ to 
an injection $\mathcal{O}\to \mathcal{O}^{\oplus r+3}$ by taking double duals,
we infer that $E_2^{0,0}$ contains a torsion sheaf $\mathcal{O}_M$.
This is a contradiction. Hence $b=0$, and $E_2^{0,0}$ fits in the following exact sequences:
\[0\to \mathcal{S}^{\vee}
\to\mathcal{O}^{\oplus r+3}\to E_2^{0,0}\to 0;\]
\[0\to \mathcal{I}_L(-1)\to E_2^{0,0}\to \mathcal{E}\to 0.\]
Since $\mathcal{I}_L(-1)$ is not locally free, so is $E_2^{0,0}$.
Hence 
the former exact sequence together with \eqref{SSdual} implies that 
$E_2^{0,0}\cong \mathcal{I}_M(1)\oplus \mathcal{O}^{\oplus r}$ for some line 
$M$ in $\mathbb{Q}^3$.
By taking the double dual of the injection $\mathcal{I}_L(-1)\to E_2^{0,0}$
in the latter exact sequence,
we obtain a commutative 
diagram with exact rows
\[
\xymatrix{
0\ar[r]&\mathcal{I}_L(-1)\ar[d]\ar[r] &E_2^{0,0}\ar[d]\ar[r] 
&\mathcal{E}\ar[r] \ar[d]   & 0    \\
0\ar[r]&\mathcal{O}(-1) \ar[r]      
&\mathcal{O}(1)\oplus \mathcal{O}^{\oplus r}\ar[r] &\mathcal{F}
\ar[r]              & 0    
}
\]
for some coherent sheaf $\mathcal{F}$.
Note that $\Tor_q^{\mathcal{O}_p}(\mathcal{F}_p,k(p))=0$ for $q\geq 2$
and any point $p$.
Since $\mathcal{E}$ is torsion-free, 
the snake lemma implies that $L=M$ and that 
we have an exact sequence
\[0\to \mathcal{O}_L(-1)\to \mathcal{O}_M(1)
\to \mathcal{O}_Z\to 0\]
for some closed subscheme 
$Z$ of length two.
Moreover $\mathcal{E}$ fits in the following exact sequence:
\[0\to \mathcal{E}\to \mathcal{F}\to \mathcal{O}_Z\to 0.\]
For an associated point $p$ of $Z$, the exact sequence above induces
a coherent sheaf $\mathcal{G}$ and 
the following exact sequence:
\[
0\to \mathcal{E}\to \mathcal{G}\to k(p)\to 0.
\]
Since $\Tor_3^{\mathcal{O}_p}(\mathcal{F}_p,k(p))=0$,
we have $\Tor_3^{\mathcal{O}_p}(\mathcal{G}_p,k(p))=0$.
Note that  
$\Tor_q^{\mathcal{O}_p}(\mathcal{E}_p,k(p))=0$ for $q\geq 1$.
Hence 
$\Tor_3^{\mathcal{O}_p}(k(p),k(p))=0$,
which contradicts the fact that $\Tor_3^{\mathcal{O}_p}(k(p),k(p))=1$.
Therefore $a$ cannot be positive: $a=0$.
Thus $0=E_2^{-1,1}=E_3^{-1,1}$, 
$0=E_2^{-1,0}
=
E_2^{-3,1}$,
$\mathcal{O}^{\oplus r+3}\cong E_2^{0,0}$, $E_3^{-3,2}\cong E_2^{-3,2}\cong \mathcal{O}(-1)$,
$E_4^{0,0}\cong \mathcal{E}$, and we have the following exact sequences:
\[0\to\mathcal{O}(-1)^{\oplus 2}\to T_{\mathbb{P}^4}(-2)|_{\mathbb{Q}^3}
\to E_2^{-2,1}\to 0;\]
\[0\to E_2^{-2,1}\to \mathcal{O}^{\oplus r+3}\to E_3^{0,0}\to 0;\]
\[0\to \mathcal{O}(-1)\to E_3^{0,0}\to \mathcal{E}\to 0.\]
Since $T_{\mathbb{P}^4}(-2)|_{\mathbb{Q}^3}$ fits in 
an exact sequence
\[
0\to \mathcal{O}(-2)\to \mathcal{O}(-1)^{\oplus 5}\to T_{\mathbb{P}^4}(-2)|_{\mathbb{Q}^3}\to 0,
\]
$E_2^{-2,1}$ has a resolution of the following form:
\[
0\to \mathcal{O}(-2)\to \mathcal{O}(-1)^{\oplus 3}\to E_2^{-2,1}\to 0.
\]
Therefore we see that $\mathcal{E}$ belongs to Case (9) of Theorem~\ref{Chern2}.

\section{The case where $\mathcal{E}|_{\mathbb{Q}^2}$ belongs to Case
(6) of Theorem~\ref{Chern(2,2)}}

Suppose that $\mathcal{E}|_{\mathbb{Q}^2}$
fits in one of the following exact 
sequence:
\[0\to \mathcal{O}^{\oplus 2}\to \mathcal{O}(1,0)^{\oplus 2}\oplus
\mathcal{O}(0,1)^{\oplus 2}\oplus  
\mathcal{O}^{\oplus r-2}\to \mathcal{E}|_{\mathbb{Q}^2}\to 0.\]
Then $h^q(\mathcal{E}(-1)|_{\mathbb{Q}^2})=0$ for any $q$,
and $c_2h=4$.
Since $h^q(\mathcal{E}|_{\mathbb{Q}^2})=0$ for any $q>0$,
the vanishing \eqref{firstvanishing} shows that $h^q(\mathcal{E}(-t))=0$ for $q\geq 2$ and $t=1,2$.
The assumption~\eqref{h^0(E(-2))vanish} together with \eqref{e(-2)RR} and \eqref{c3nonnegative}
shows that 
\[0\geq -h^1(\mathcal{E}(-2))=\chi(\mathcal{E}(-2))=-1+\frac{1}{2}c_3\geq -1.\]
Therefore we have two cases:
$(h^1(\mathcal{E}(-2)),c_3)=(0,2)$ or $(1,0)$.
Note here that $h^q(\mathcal{E}(-1))=h^q(\mathcal{E}(-2))$ for any $q$.
In particular, $h^0(\mathcal{E}(-1))=h^0(\mathcal{E}(-2))=0$ by
\eqref{h^0(E(-2))vanish}.

We claim here that $h^q(\mathcal{S}^{\vee}\otimes \mathcal{E}(t)|_{\mathbb{Q}^2})=0$
for $q>0$ and $t\geq 0$.
Indeed, we see that 
\[h^q((\mathcal{O}(-1,0)\oplus\mathcal{O}(0,-1))\otimes
(\mathcal{O}(1+t,t)^{\oplus 2}\oplus\mathcal{O}(t,1+t)^{\oplus 2}
\oplus \mathcal{O}(t,t)^{\oplus r-3}))=0\]
for $q>0$ and $t\geq 0$.
Hence we obtain the claim.
Then it follows from \eqref{ExtSE(1)q2ijouvanishing}
that 
\begin{equation}\label{(3-1)casevanishing}
h^q(\mathcal{S}^{\vee}\otimes\mathcal{E}(t))=0\textrm{ for }q\geq 2
\textrm{ and }t=0, -1.
\end{equation}
Since $h^0(\mathcal{E}(-1))=0$,
the exact sequence \eqref{SSdualE(-1)twist} together with 
\eqref{(3-1)casevanishing}
shows that $h^q(\mathcal{S}^{\vee}\otimes \mathcal{E}(-1))=0$
unless $q=1$.
Hence 
\begin{equation}\label{(3-1)h1SdualE(-1)noc3niyorubaaiwake}
-h^1(\mathcal{S}^{\vee}\otimes\mathcal{E}(-1))
=\chi(\mathcal{S}^{\vee}\otimes\mathcal{E}(-1))
=-4+c_3
\end{equation}
by \eqref{SERR(-1)}.

\subsection{Suppose that $(h^1(\mathcal{E}(-2)),c_3)=(0,2)$.} 
Then $h^q(\mathcal{E}(-2))=0$ for any $q$.
Hence $h^q(\mathcal{E}(-1))=0$ for any $q$.
Set $a=h^1(\mathcal{E}(-2)|_{\mathbb{Q}^2})$.
Then $a\leq 2$ and $h^2(\mathcal{E}(-2)|_{\mathbb{Q}^2})=r-4+a$.
Thus $h^2(\mathcal{E}(-3))=a$, $h^3(\mathcal{E}(-3))=r-4+a$, 
and $h^q(\mathcal{E}(-3))=0$ unless $q=2$ or $3$.
It follows from \eqref{(3-1)h1SdualE(-1)noc3niyorubaaiwake}
that $h^1(\mathcal{S}^{\vee}\otimes\mathcal{E}(-1))=2$.
We apply to 
$\mathcal{E}(-1)$
the Bondal spectral sequence~\eqref{BondalSpectral}.
We have $\Ext^3(G,\mathcal{E}(-1))\cong S_3^{\oplus r-4+a}$,
$\Ext^2(G,\mathcal{E}(-1))\cong S_3^{\oplus a}$, 
$\Ext^1(G,\mathcal{E}(-1))\cong S_1^{\oplus 2}$,
and $\Hom(G,\mathcal{E}(-1))=0$.
Lemma~\ref{S2Arilemma} then shows that 
$E_2^{-3,3}\cong \mathcal{O}(-1)^{\oplus r-4+a}$,
that $E_2^{-3,2}\cong \mathcal{O}(-1)^{\oplus a}$,
that $E_2^{-1,1}\cong \mathcal{S}(-1)^{\oplus 2}$,
and $E_2^{p,q}=0$ unless $(p,q)=(-3,3)$, $(-3,2)$, or $(-1,1)$.
Then $\mathcal{E}(-1)$ fits in the $(-1)$-twist of the following exact sequence:
\begin{equation}\label{(2,0)(0,2)ortwo(0,1)baaiwake}
0\to \mathcal{O}^{\oplus a}
\to \mathcal{S}^{\oplus 2}\to \mathcal{E}\to \mathcal{O}^{\oplus r-4+a}
\to 0.
\end{equation}
This sequence
splits into the following two exact sequences:
\[
0\to \mathcal{O}^{\oplus a}
\to \mathcal{S}^{\oplus 2}\to \mathcal{F}\to 0;\]
\[0\to \mathcal{F}\to \mathcal{E}\to \mathcal{O}^{\oplus r-4+a}
\to 0,\]
where $\mathcal{F}$ is a globally generated vector bundle of rank $4-a$.
We claim here that $a\leq 1$.
Indeed, if $a=2$, then 
we have the following exact sequences:
\[0\to \mathcal{O}
\to \mathcal{S}^{\oplus 2}\to \mathcal{G}\to 0;\]
\[0\to \mathcal{O}\to \mathcal{G}\to \mathcal{F}\to 0,\]
where $\mathcal{G}$ is a globally generated vector bundle 
of rank $3$.
Since $\mathcal{F}$ is a vector bundle, $\mathcal{G}$ must have a nowhere vanishing global section,
and thus $c_3(\mathcal{G})=0$.
On the other hand, $c_3(\mathcal{G})=c_3(\mathcal{S}^{\oplus 2})=2c_2(\mathcal{S})h=2$.
This is a contradiction. Hence the case $a=2$ does not arise.
Now note that 
$\mathcal{E}$ is isomorphic to 
$\mathcal{F}\oplus \mathcal{O}^{\oplus r-4+a}$
since $h^1(\mathcal{F})=0$.
Therefore $\mathcal{E}$ fits in an exact sequence
\[
0\to \mathcal{O}^{\oplus a}\to \mathcal{S}^{\oplus 2}\oplus \mathcal{O}^{\oplus r-4+a}
\to \mathcal{E}\to 0,
\]
where the composite of the inclusion 
$\mathcal{O}^{\oplus a}\to \mathcal{S}^{\oplus 2}\oplus \mathcal{O}^{\oplus r-4+a}$
and the projection
$\mathcal{S}^{\oplus 2}\oplus \mathcal{O}^{\oplus r-4+a}
\to
\mathcal{O}^{\oplus r-4+a}$
is zero.
This is Case (5) of Theorem~\ref{Chern2}.

\subsection{Suppose that $(h^1(\mathcal{E}(-2)),c_3)=(1,0)$.}
Then $h^1(\mathcal{E}(-1))=1$. 
Hence $h^0(\mathcal{E})=h^0(\mathcal{E}|_{\mathbb{Q}^2})-1=r+3$.
It follows from \eqref{(3-1)h1SdualE(-1)noc3niyorubaaiwake}
that $h^1(\mathcal{S}^{\vee}\otimes\mathcal{E}(-1))=4$.
Set $a=h^0(\mathcal{S}^{\vee}\otimes \mathcal{E})$.
Then the exact sequence \eqref{SSdualE(-1)twist}
shows that $a\leq 4$, that $h^q(\mathcal{S}^{\vee}\otimes \mathcal{E})=0$
unless $q=0$ or $1$, and that $h^1(\mathcal{S}^{\vee}\otimes\mathcal{E})=a$.
Hence we have $\Ext^q(G,\mathcal{E})=0$ for $q=2$ and $3$,
and $\Hom(G,\mathcal{E})$ fits in an exact sequence
\[0\to S_0^{\oplus r+3}\to \Hom(G,\mathcal{E})\to S_1^{\oplus a}\to 0.\]
Moreover $\Ext^1(G,\mathcal{E})$ has a filtration $S_1^{\oplus a}\subset F\subset \Ext^1(G,\mathcal{E})$
of right $A$-modules such that the following sequences are exact:
\[
0\to F\to \Ext^1(G,\mathcal{E})\to S_3\to 0;\]
\[
0\to S_1^{\oplus a}\to F\to S_2\to 0.\]
Now the structures of right $A$-modules $\Ext^q(G,\mathcal{E})$'s  are the same as 
those of $\Ext^q(G,\mathcal{E})$'s in 
Section~\ref{NewCase(5)OfTh2.2h^0(E(-1))=0c3=0},
and we conclude that $\mathcal{E}$ belongs to Case (9) of Theorem~\ref{Chern2}.

\section{The case where $\mathcal{E}|_{\mathbb{Q}^2}$ belongs to Case
(7) of Theorem~\ref{Chern(2,2)}}
Suppose that $\mathcal{E}|_{\mathbb{Q}^2}$ fits in 
the following exact sequence:
\[
0\to \mathcal{O}(-1,-1)
\oplus \mathcal{O}(-1,0)
\oplus \mathcal{O}(0,-1)
\to 
\mathcal{O}^{\oplus r+3}\to \mathcal{E}|_{\mathbb{Q}^2}\to 0.
\]
Then $c_2h=5$.
It then follows from \eqref{tautologicalSelfIntersectionNonNegative}
that $c_3\geq 4$.
Note that 
\begin{equation}\label{SdualErestrictionInCase9}
h^q(\mathcal{S}^{\vee}\otimes \mathcal{E}|_{\mathbb{Q}^2})
=
\begin{cases}
2& \textrm{if}\quad  q= 0\\
0& \textrm{if}\quad  q\neq 0,
\end{cases}
\end{equation}
and that
\begin{equation}\label{CohomologyOfE(-1)RestInCase9InTh2-1}
h^q(\mathcal{E}(-1)|_{\mathbb{Q}^2})
=
\begin{cases}
1& \textrm{if}\quad  q= 1\\
0& \textrm{if}\quad  q\neq 1.
\end{cases}
\end{equation}
Hence we have 
\[h^0(\mathcal{E}(-1))=0\]
by \eqref{h^0(E(-2))vanish}.
Note that 
$H^q(\mathcal{E}|_{\mathbb{Q}^2})$ vanishes for any $q>0$.
Since $h^q(\mathcal{E})=0$ for any $q>0$ by \eqref{firstvanishing},
we have $h^q(\mathcal{E}(-1))=0$ for any $q\geq 2$.
Hence 
it follows from \eqref{e(-1)RR} that  
\[0\geq  -h^1(\mathcal{E}(-1))=\chi(\mathcal{E}(-1))=-\dfrac{5}{2}+\dfrac{1}{2}c_3
\geq -\dfrac{1}{2}.\]
Therefore $c_3=5$ and $h^1(\mathcal{E}(-1))=0$.
Now it follows from \eqref{SSdualE(-1)twist}  that
$h^q(\mathcal{S}^{\vee}\otimes\mathcal{E})=h^{q+1}(\mathcal{S}^{\vee}\otimes\mathcal{E}(-1))$
for any $q$.
In particular,  $h^0(\mathcal{S}^{\vee}\otimes \mathcal{E}(-1))=0$.
Moreover 
$h^{q+1}(\mathcal{S}^{\vee}\otimes\mathcal{E}(-1))
=h^{q+1}(\mathcal{S}^{\vee}\otimes\mathcal{E})$
for $q\geq 1$
by \eqref{SdualErestrictionInCase9} and \eqref{SdualE(-1)SdualEtoRestriction}.
Hence $h^q(\mathcal{S}^{\vee}\otimes\mathcal{E})=0$ for $q\geq 1$
and $h^{q}(\mathcal{S}^{\vee}\otimes\mathcal{E}(-1))=0$ for $q\geq 2$.
Therefore 
\[
-h^1(\mathcal{S}^{\vee}\otimes\mathcal{E}(-1))
=\chi(\mathcal{S}^{\vee}\otimes\mathcal{E}(-1))
=-6+c_3=-1
\]
by \eqref{SERR(-1)}.
Thus 
$h^0(\mathcal{S}^{\vee}\otimes \mathcal{E})=1$.
We apply to $\mathcal{E}$ the Bondal spectral sequence~\eqref{BondalSpectral}.
From 
\eqref{CohomologyOfE(-1)RestInCase9InTh2-1},
it follows that 
$h^q(\mathcal{E}(-2))=0$ unless $q=2$
and that $h^2(\mathcal{E}(-2))=1$. 
Since $h^0(\mathcal{E}|_{\mathbb{Q}^2})=r+3$, we see that $h^0(\mathcal{E})=r+3$.
Hence we have 
an exact sequence
\[
0\to S_0^{\oplus r+3}\to \Hom(G,\mathcal{E})\to S_1\to 0,
\]
and the following:
$\Ext^q(G,\mathcal{E})
=
0$ for $q=1,3$;
$\Ext^2(G,\mathcal{E})\cong S_3$.
Therefore Lemma~\ref{S2Arilemma} implies 
that $E_2^{p,q}=0$ unless $(p,q)=(-3,2)$ or $(0,0)$,
that $E_2^{-3,2}
\cong 
\mathcal{O}(-1)$,
and that there is  the following exact sequence:
\[
0\to \mathcal{S}(-1)\to \mathcal{O}^{\oplus r+3}
\to E_2^{0,0}\to 0.\]
Note that we have 
the following exact sequence:
\[0\to E_2^{-3,2}\to E_2^{0,0}\to \mathcal{E}\to 0.\]
Since $\Ext^1(\mathcal{O}(-1),\mathcal{S}(-1))=0$, this implies that 
$\mathcal{E}$ fits in the following exact sequence:
\[
0\to \mathcal{S}(-1)\oplus \mathcal{O}(-1)
\to \mathcal{O}^{\oplus r+3}\to \mathcal{E}\to 0.
\]
This is Case (6) of Theorem~\ref{Chern2}.

\section{The case where $\mathcal{E}|_{\mathbb{Q}^2}$ belongs to Case
(8) of Theorem~\ref{Chern(2,2)}}
Suppose that $\mathcal{E}|_{\mathbb{Q}^2}$ fits in either of the following exact sequences:
\[
0\to \mathcal{O}(-1,-2)
\to 
\mathcal{O}(1,0)\oplus
\mathcal{O}^{\oplus r}\to \mathcal{E}|_{\mathbb{Q}^2}\to 0;
\]
Then $c_2h=6$.
It then follows from \eqref{tautologicalSelfIntersectionNonNegative}
that $c_3\geq 8$.
Note that
\begin{equation}\label{CohomologyOfE(-1)RestInCase11InTh2-1}
h^q(\mathcal{E}(-1)|_{\mathbb{Q}^2})
=
\begin{cases}
2& \textrm{if}\quad  q= 1\\
0& \textrm{if}\quad  q\neq 1,
\end{cases}
\end{equation}
and that
\begin{equation}\label{SdualErestrictionInCase11}
h^q(\mathcal{S}^{\vee}\otimes \mathcal{E}|_{\mathbb{Q}^2})
=
\begin{cases}
1& \textrm{if}\quad  q=0, 1\\
0& \textrm{if}\quad  q\neq 0, 1.
\end{cases}
\end{equation}
Hence we have 
\[h^0(\mathcal{E}(-1))=0\]
by \eqref{h^0(E(-2))vanish}.
Note that 
$H^q(\mathcal{E}|_{\mathbb{Q}^2})$ vanishes for all $q>0$.
Since 
$h^q(\mathcal{E})=0$ for all $q>0$ by \eqref{firstvanishing},
we have $h^q(\mathcal{E}(-1))=0$ for all $q\geq 2$.
It follows from \eqref{e(-1)RR} that  
\[0\geq -h^1(\mathcal{E}(-1))=\chi(\mathcal{E}(-1))=-4+\dfrac{1}{2}c_3
\geq 0.\]
Therefore $c_3=8$
and $h^1(\mathcal{E}(-1))=0$.
Now it follows from \eqref{SSdualE(-1)twist}  that
$h^q(\mathcal{S}^{\vee}\otimes\mathcal{E})=h^{q+1}(\mathcal{S}^{\vee}\otimes\mathcal{E}(-1))$
for any $q$.
In particular,  $h^0(\mathcal{S}^{\vee}\otimes \mathcal{E}(-1))=0$.
Moreover 
$h^{q+1}(\mathcal{S}^{\vee}\otimes\mathcal{E}(-1))
=h^{q+1}(\mathcal{S}^{\vee}\otimes\mathcal{E})$
for $q\geq 2$
by \eqref{SdualE(-1)SdualEtoRestriction}
and \eqref{SdualErestrictionInCase11}.
Hence $h^q(\mathcal{S}^{\vee}\otimes\mathcal{E})=0$ for $q\geq 2$
and $h^{3}(\mathcal{S}^{\vee}\otimes\mathcal{E}(-1))=0$.
Hence 
\[
-h^1(\mathcal{S}^{\vee}\otimes\mathcal{E}(-1))+
h^2(\mathcal{S}^{\vee}\otimes\mathcal{E}(-1))
=\chi(\mathcal{S}^{\vee}\otimes\mathcal{E}(-1))
=-8+c_3=0
\]
by \eqref{SERR(-1)}.
Set $a=h^0(\mathcal{S}^{\vee}\otimes\mathcal{E})$.
Then $a=h^1(\mathcal{S}^{\vee}\otimes\mathcal{E}(-1))
=h^2(\mathcal{S}^{\vee}\otimes\mathcal{E}(-1))
=h^1(\mathcal{S}^{\vee}\otimes\mathcal{E})$.
We see that $a=1$ by  \eqref{SdualE(-1)SdualEtoRestriction}
and \eqref{SdualErestrictionInCase11}.
We 
apply to $\mathcal{E}$ the Bondal spectral sequence~\eqref{BondalSpectral}.
It follows from \eqref{CohomologyOfE(-1)RestInCase11InTh2-1}
that $h^q(\mathcal{E}(-2))$ vanishes 
unless $q=2$
and that $h^2(\mathcal{E}(-2))=2$. 
Since $h^0(\mathcal{E}|_{\mathbb{Q}^2})=r+2$, we see that $h^0(\mathcal{E})=r+2$.
Therefore $\Ext^3(G,\mathcal{E})=0$,
$\Ext^2(G,\mathcal{E})\cong S_3^{\oplus 2}$,
$\Ext^1(G,\mathcal{E})
\cong 
S_1$, 
and 
$\Hom(G,\mathcal{E})$ fits in the following exact sequence:
\[
0\to S_0^{\oplus r+2}\to \Hom(G,\mathcal{E})\to S_1\to 0.
\]
Therefore Lemma~\ref{S2Arilemma} implies 
that $E_2^{p,q}=0$ unless $(p,q)=(-3,2)$, $(-1,1)$, $(-1,0)$, or $(0,0)$,
that $E_2^{-3,2}\cong \mathcal{O}(-1)^{\oplus 2}$,
that $E_2^{-1,1}\cong \mathcal{S}(-1)$,
and that there exists the following exact sequence:
\[0\to E_2^{-1,0}\to \mathcal{S}(-1)\to \mathcal{O}^{\oplus r+2}\to E_2^{0,0}\to 0.\]
The Bondal spectral sequence implies that $E_2^{-1,0}=0$,
that $E_2^{0,0}\cong E_3^{0,0}$,
and that we have the following exact sequences:
\[
0\to E_3^{-3,2}\to \mathcal{O}(-1)^{\oplus 2}\xrightarrow{\varphi} \mathcal{S}(-1)\to E_3^{-1,1}\to 0;
\]
\[0\to E_3^{-3,2}\to E_3^{0,0}\to E_4^{0,0}\to 0;\]
\[
0\to E_4^{0,0}\to \mathcal{E}\to E_3^{-1,1}\to 0.
\]
Since $\mathcal{E}$ is nef, $E_3^{-1,1}$ cannot admit a negative degree quotient.
Hence $\varphi\neq 0$.
Thus the composite of some inclusion $\mathcal{O}(-1)\to \mathcal{O}(-1)^{\oplus 2}$
and $\varphi$ is non-zero, and it induces a morphism 
$\bar{\varphi}:\mathcal{O}(-1)\to \mathcal{I}_L$ for some line $L$ in $\mathbb{Q}^3$.
Moreover $\bar{\varphi}$ fits in the following exact sequence:
\[
0\to E_3^{-3,2}\to \mathcal{O}(-1)\xrightarrow{\bar{\varphi}} \mathcal{I}_L\to E_3^{-1,1}\to 0.
\]
This shows that $E_3^{-1,1}|_M$ admits a negative degree quotient for some line $M$ in $\mathbb{Q}^3$.
This is a contradiction. Therefore $\mathcal{E}|_{\mathbb{Q}^2}$ cannot belong to Case 
(8) of Theorem~\ref{Chern(2,2)}.

\section{The case where $\mathcal{E}|_{\mathbb{Q}^2}$ belongs to Case
(9) of Theorem~\ref{Chern(2,2)}}
Suppose that $\mathcal{E}|_{\mathbb{Q}^2}$ fits in either of the following exact sequences:
\[
0\to \mathcal{O}(-1,-1)^{\oplus 2}
\to 
\mathcal{O}^{\oplus r+2}\to \mathcal{E}|_{\mathbb{Q}^2}\to 0.
\]
Then $c_2h=6$.
It then follows from \eqref{tautologicalSelfIntersectionNonNegative}
that $c_3\geq 8$.
Note that
\begin{equation}\label{CohomologyOfE(-1)RestInNewCase10InTh2-1}
h^q(\mathcal{E}(-1)|_{\mathbb{Q}^2})
=
\begin{cases}
2& \textrm{if}\quad  q= 1\\
0& \textrm{if}\quad  q\neq 1,
\end{cases}
\end{equation}
and 
that
\begin{equation}\label{SdualErestrictionInCase12}
h^q(\mathcal{S}^{\vee}\otimes \mathcal{E}|_{\mathbb{Q}^2})
=0 \textrm{ for all } q.
\end{equation}
Hence we have 
\[h^0(\mathcal{E}(-1))=0\]
by \eqref{h^0(E(-2))vanish}.
Note that 
$H^q(\mathcal{E}|_{\mathbb{Q}^2})$ vanishes for all $q>0$.
Since 
$h^q(\mathcal{E})=0$ for all $q>0$ by \eqref{firstvanishing},
we have $h^q(\mathcal{E}(-1))=0$ for all $q\geq 2$.
It follows from \eqref{e(-1)RR} that  
\[0\geq -h^1(\mathcal{E}(-1))=\chi(\mathcal{E}(-1))=-4+\dfrac{1}{2}c_3
\geq 0.\]
Therefore $c_3=8$
and $h^1(\mathcal{E}(-1))=0$.
Now it follows from \eqref{SSdualE(-1)twist}  that
$h^q(\mathcal{S}^{\vee}\otimes\mathcal{E})=h^{q+1}(\mathcal{S}^{\vee}\otimes\mathcal{E}(-1))$
for any $q$.
Moreover 
$h^{q+1}(\mathcal{S}^{\vee}\otimes\mathcal{E}(-1))
=h^{q+1}(\mathcal{S}^{\vee}\otimes\mathcal{E})$
for any $q$
by \eqref{SdualE(-1)SdualEtoRestriction}
and \eqref{SdualErestrictionInCase12}.
Hence $h^q(\mathcal{S}^{\vee}\otimes\mathcal{E})=0$ for any $q$.
We apply to $\mathcal{E}$ the Bondal spectral sequence~\eqref{BondalSpectral}.
It follows from \eqref{CohomologyOfE(-1)RestInNewCase10InTh2-1}
that $h^q(\mathcal{E}(-2))$ vanishes 
unless $q=2$
and that $h^2(\mathcal{E}(-2))=2$. 
Since $h^0(\mathcal{E}|_{\mathbb{Q}^2})=r+2$, we see that $h^0(\mathcal{E})=r+2$.
Therefore $\Hom(G,\mathcal{E})\cong S_0^{\oplus r+2}$,
$\Ext^1(G,\mathcal{E})=0$, 
$\Ext^2(G,\mathcal{E})\cong S_3^{\oplus 2}$,
and 
$\Ext^3(G,\mathcal{E})=0$.
Therefore Lemma~\ref{S2Arilemma} implies 
that $E_2^{p,q}=0$ unless $(p,q)=(-3,2)$ or $(0,0)$,
that $E_2^{-3,2}\cong \mathcal{O}(-1)^{\oplus 2}$,
and that $E_2^{0,0}\cong \mathcal{O}^{\oplus r+2}$.
It follows from  the Bondal spectral sequence
that 
$\mathcal{E}$ fits in the following exact sequence:
\[
0\to \mathcal{O}(-1)^{\oplus 2}
\to \mathcal{O}^{\oplus r+2}\to \mathcal{E}\to 0.
\]
This is Case (7) of Theorem~\ref{Chern2}.

\section{The case where $\mathcal{E}|_{\mathbb{Q}^2}$ belongs to Case
(10) of Theorem~\ref{Chern(2,2)}}
Suppose that $\mathcal{E}|_{\mathbb{Q}^2}$ fits in the following exact sequence:
\[
0\to \mathcal{O}(-2,-2)
\to 
\mathcal{O}^{\oplus r+1}\to \mathcal{E}|_{\mathbb{Q}^2}\to 0.
\]
Then $c_2h=8$.
It then follows from \eqref{tautologicalSelfIntersectionNonNegative}
that $c_3\geq 16$.
Note that
\begin{equation}\label{CohomologyOfEestInCase13InTh2-1}
h^q(\mathcal{E}|_{\mathbb{Q}^2})=
\begin{cases}
r+1& \textrm{if}\quad  q= 0\\
1& \textrm{if}\quad  q= 1\\
0& \textrm{if}\quad  q= 2,
\end{cases}
\end{equation}
that
\begin{equation}\label{CohomologyOfE(-1)RestInCase13InTh2-1}
h^q(\mathcal{E}(-1)|_{\mathbb{Q}^2})=
\begin{cases}
4& \textrm{if}\quad  q= 1\\
0& \textrm{if}\quad  q\neq 1,
\end{cases}
\end{equation}
that 
\begin{equation}\label{SdualE(1)restrictionInCase13}
h^q(\mathcal{S}^{\vee}\otimes \mathcal{E}(1)|_{\mathbb{Q}^2})=
\begin{cases}
4r+4& \textrm{if}\quad  q= 0\\
0& \textrm{if}\quad  q\neq 0,
\end{cases}
\end{equation}
and that 
\begin{equation}\label{SdualErestrictionInCase13}
h^q(\mathcal{S}^{\vee}\otimes \mathcal{E}|_{\mathbb{Q}^2})=
\begin{cases}
4& \textrm{if}\quad  q= 1\\
0& \textrm{if}\quad  q\neq 1.
\end{cases}
\end{equation}
Hence we have 
\[h^0(\mathcal{E}(-1))=0\]
by \eqref{h^0(E(-2))vanish}.
Then $h^0(\mathcal{S}^{\vee}\otimes \mathcal{E}(-1))=0$ 
by  \eqref{SSdualE(-1)twist}.
Since 
$h^q(\mathcal{E})=0$ for all $q>0$ by \eqref{firstvanishing},
we have 
$h^2(\mathcal{E}(-1))=1$ and $h^3(\mathcal{E}(-1))=0$
by \eqref{CohomologyOfEestInCase13InTh2-1}.
It then follows from \eqref{e(-1)RR} that  
\[1\geq 1 -h^1(\mathcal{E}(-1))=\chi(\mathcal{E}(-1))=-7+\dfrac{1}{2}c_3
\geq 1.\]
Therefore $c_3=16$
and $h^1(\mathcal{E}(-1))=0$. Hence $h^0(\mathcal{E})=r+1$
since $h^0(\mathcal{E}|_{\mathbb{Q}^2})=r+1$ by 
\eqref{CohomologyOfEestInCase13InTh2-1}.
Moreover $h^2(\mathcal{E}(-2))=5$ and $h^q(\mathcal{E}(-2))=0$ unless $q=2$ 
by \eqref{CohomologyOfE(-1)RestInCase13InTh2-1}.
It follows from \eqref{ExtSE(1)q2ijouvanishing}
and \eqref{SdualE(1)restrictionInCase13} that 
\[
h^q(\mathcal{S}^{\vee}\otimes\mathcal{E})=0\textrm{ for }q\geq 2.
\]
Moreover $h^0(\mathcal{S}^{\vee}\otimes\mathcal{E})=0$ 
since $h^0(\mathcal{S}^{\vee}\otimes\mathcal{E}|_{\mathbb{Q}^2})=0$
by \eqref{SdualErestrictionInCase13}.
Hence it follows from \eqref{SERR}
\[
-h^1(\mathcal{S}^{\vee}\otimes\mathcal{E})
=\chi(\mathcal{S}^{\vee}\otimes\mathcal{E})
=16-4c_2h+c_3=0.
\]
We apply to $\mathcal{E}$ the Bondal spectral sequence~\eqref{BondalSpectral}.
We see that $\Hom(G,\mathcal{E})\cong S_0^{\oplus r+1}$,
that $\Ext^q(G,\mathcal{E})=0$ for $q=1,3$,
and that 
$\Ext^2(G,\mathcal{E})$ fits in the following exact sequence 
of right $A$-modules:
\[0\to S_2\to \Ext^2(G,\mathcal{E})\to S_3^{\oplus 5}\to 0.\]
Therefore Lemma~\ref{S2Arilemma} implies 
that $E_2^{p,q}=0$ unless $(p,q)=(-3,2)$ $(-2,2)$ or $(0,0)$,
that $E_2^{0,0}\cong \mathcal{O}^{\oplus r+1}$,
and 
that $E_2^{-3,2}$ and $E_2^{-2,2}$ fit in the following exact sequence:
\begin{equation}\label{(13)noE_2^22noMotonoExSeq}
0\to E_2^{-3,2}\to \mathcal{O}(-1)^{\oplus 5}\to T_{\mathbb{P}^4}(-2)|_{\mathbb{Q}^3}
\to E_2^{-2,2}\to 0.
\end{equation}
The Bondal spectral sequence induces 
the following 
isomorphisms and exact sequences:
\[E_2^{-3,2}\cong E_3^{-3,2};\]
\[E_2^{0,0}\cong E_3^{0,0};\]
\[
0\to E_3^{-3,2}\to E_3^{0,0}\to E_4^{0,0}\to 0;
\]
\[
0\to E_4^{0,0}\to \mathcal{E}\to E_2^{-2,2}\to 0.
\]
Note here that $E_2^{-2,2}|_L$ cannot admit a negative degree quotient 
for any line $L\subset\mathbb{Q}^3$ since $\mathcal{E}$ is nef.
We will show that $E_2^{-2,2}=0$; 
first note that the exact sequence \eqref{(13)noE_2^22noMotonoExSeq}
induces the following exact sequence:
\[
0\to E_2^{-3,2}\to \mathcal{O}(-1)^{\oplus 5}
\oplus \mathcal{O}(-2)
\xrightarrow{p}
\mathcal{O}(-1)^{\oplus 5}
\to E_2^{-2,2}\to 0.
\]
Consider the composite of the inclusion 
$\mathcal{O}(-1)^{\oplus 5}\to  \mathcal{O}(-1)^{\oplus 5}
\oplus \mathcal{O}(-2)$ and the morphism $p$ above,
and let $\mathcal{O}(-1)^{\oplus a}$ be the cokernel of this composite.
Then we have the following exact sequence:
\[
\mathcal{O}(-2)\xrightarrow{\pi} \mathcal{O}(-1)^{\oplus a}\to E_2^{-2,2}\to 0.
\]
We claim here that $a=0$. Suppose, to the contrary, that $a>0$.
Since $E_2^{-2,2}$ cannot be isomorphic to $\mathcal{O}(-1)^{\oplus a}$,
the morphism $\pi$ above
is not zero.
Therefore the composite of $\pi$ and some projection
$\mathcal{O}(-1)^{\oplus a}\to \mathcal{O}(-1)$ is not zero,
whose quotient is of the form $\mathcal{O}_H(-1)$ for some hyperplane 
$H$ in $\mathbb{Q}^3$.
Hence $E_2^{-2,2}$ admits $\mathcal{O}_H(-1)$ as a quotient. This is a contradiction.
Thus $a=0$ and $E_2^{-2,2}=0$. Moreover we see that $E_2^{-3,2}\cong \mathcal{O}(-2)$.
Therefore $\mathcal{E}$ fits in the following exact sequence:
\[
0\to \mathcal{O}(-2)\to \mathcal{O}^{\oplus r+1}\to \mathcal{E}\to 0.
\]
This is Case (8) of Theorem~\ref{Chern2}.

\section{The case where $\mathcal{E}|_{\mathbb{Q}^2}$ belongs to Case
(11) of Theorem~\ref{Chern(2,2)}}
Suppose that $\mathcal{E}|_{\mathbb{Q}^2}$ fits in the following exact sequence:
\[
0\to \mathcal{O}(-2,-2)
\to 
\mathcal{O}^{\oplus r+1}\to \mathcal{E}|_{\mathbb{Q}^2}\to k(p)\to  0.
\]
Then $c_2h=7$.
It then follows from \eqref{tautologicalSelfIntersectionNonNegative}
that 
\[c_3\geq 12.\]
We claim here that $h^0(\mathcal{E}(-1)|_{\mathbb{Q}^2})=0$.
Indeed, if $h^0(\mathcal{E}(-1)|_{\mathbb{Q}^2})\neq 0$, then 
\[c_2h\leq 
c_1(\mathcal{E}|_{\mathbb{Q}^2})
(c_1(\mathcal{E}|_{\mathbb{Q}^2})-c_1(\mathcal{O}_{\mathbb{Q}^2}(1,1)))=4\]
by \cite[Lemma~10.1]{MR4453350}. This is a contradiction. 
Hence $h^0(\mathcal{E}(-1)|_{\mathbb{Q}^2})=0$.
Thus we have $h^0(\mathcal{E}(-1))=0$
by \eqref{h^0(E(-2))vanish}.
It follows from \eqref{e(-1)RR} that  
\[\chi(\mathcal{E}(-1))=-\dfrac{11}{2}+\dfrac{1}{2}c_3.\]
In particular $c_3$ is odd, and thus $c_3>12$.
Therefore $h^q(\mathcal{E}(-1))=0$ for all $q>0$ by 
\eqref{higherE(-1)vanish}.
This implies that $\chi(\mathcal{E}(-1))=0$, which is a contradiction.
Therefore $\mathcal{E}|_{\mathbb{Q}^2}$ cannot belong to Case (11) of Theorem~\ref{Chern(2,2)}.

\section{The case where $\mathcal{E}|_{\mathbb{Q}^2}$ belongs to Case
(12) or (13) of Theorem~\ref{Chern(2,2)}}\label{Case(13)or(14)OfTheorem(2,2)}
Suppose that $\mathcal{E}|_{\mathbb{Q}^2}$ fits in either of the following exact sequences:
\[
0\to \mathcal{O}(-2,-2)
\to 
\mathcal{O}^{\oplus r}\to \mathcal{E}|_{\mathbb{Q}^2}\to \mathcal{O}\to  0;
\]
\[0\to \mathcal{O}(-1,-1)^{\oplus 4}\to   
\mathcal{O}^{\oplus r}\oplus\mathcal{O}(-1,0)^{\oplus 2}
\oplus \mathcal{O}(0,-1)^{\oplus 2}
\to \mathcal{E}|_{\mathbb{Q}^2}\to 0.\]
Then $c_2h=8$.
It then follows from \eqref{tautologicalSelfIntersectionNonNegative}
that 
\[c_3\geq 16.\]
We claim here that $h^0(\mathcal{E}(-1)|_{\mathbb{Q}^2})=0$.
Indeed, if $h^0(\mathcal{E}(-1)|_{\mathbb{Q}^2})\neq 0$, then 
\[c_2h\leq 
c_1(\mathcal{E}|_{\mathbb{Q}^2})
(c_1(\mathcal{E}|_{\mathbb{Q}^2})-c_1(\mathcal{O}_{\mathbb{Q}^2}(1,1)))=4\]
by \cite[Lemma~10.1]{MR4453350}. This is a contradiction. 
Hence $h^0(\mathcal{E}(-1)|_{\mathbb{Q}^2})=0$.
Thus we have $h^0(\mathcal{E}(-1))=0$
by \eqref{h^0(E(-2))vanish}.
Note that $h^q(\mathcal{E}|_{\mathbb{Q}^2})=0$ for all $q>0$.
Since 
$h^q(\mathcal{E})=0$ for all $q>0$ by \eqref{firstvanishing},
this implies that 
$h^q(\mathcal{E}(-1))=0$ for all $q\geq 2$.
It follows from \eqref{e(-1)RR} that  
\[0\geq -h^1(\mathcal{E}(-1))=
\chi(\mathcal{E}(-1))=-7+\dfrac{1}{2}c_3\geq 1.\]
This is a contradiction.
Therefore $\mathcal{E}|_{\mathbb{Q}^2}$ cannot belong to Case (12) 
or (13) of Theorem~\ref{Chern(2,2)}.

\bibliographystyle{alpha}

\begin{thebibliography}{Ohn23b}

\bibitem[BHM14]{MR3120621}
Edoardo Ballico, Sukmoon Huh, and Francesco Malaspina.
\newblock Globally generated vector bundles of rank 2 on a smooth quadric
  threefold.
\newblock {\em J.\ Pure Appl.\ Albebra}, 218(2):197--207, 2014.

\bibitem[BHM16]{MR3509232}
Edoardo Ballico, Sukmoon Huh, and Francesco Malaspina.
\newblock On higher rank globally generated vector bundles over a smooth
  quadric threefold.
\newblock {\em Proc.\ Edinb.\ Math.\ Soc.\ (2)}, 59(2):311--337, 2016.

\bibitem[Bon89]{MR992977}
Alexey~I. Bondal.
\newblock Representations of associative algebras and coherent sheaves.
\newblock {\em Izv.\ Akad.\ Nauk SSSR Ser.\ Mat.}, 53(1):25--44, 1989.

\bibitem[Ful98]{fl}
Willian Fulton.
\newblock {\em Intersection theory}, volume~2 of {\em Ergebnisse der Mathematik
  und ihrer Grenzgebiete (3)}.
\newblock Springer-Verlag, Berlin, second edition, 1998.

\bibitem[HL10]{Huybrechts-Lehn}
Daniel Huybrechts and Manfred Lehn.
\newblock {\em The Geometry of moduli spaces of sheaves, Second Edition}.
\newblock Cambridge Math.\ Lib. Cambridge University Press, Cambridge, 2010.

\bibitem[Huy06]{MR2244106}
Daniel Huybrechts.
\newblock {\em Fourier-{M}ukai transforms in algebraic geometry}.
\newblock Oxford Mathematical Monographs. The Clarendon Press Oxford University
  Press, Oxford, 2006.

\bibitem[Kap88]{MR0939472}
M.~M. Kapranov.
\newblock On the derived categories of coherent sheaves on some homogeneous
  spaces.
\newblock {\em Invent.\ Math.}, 92(3):479--508, 1988.

\bibitem[Lan98]{MR1633159}
Adrian Langer.
\newblock Fano 4-folds with scroll structure.
\newblock {\em Nagoya Math. J.}, 150:135--176, 1998.

\bibitem[Laz04]{MR2095472}
Robert Lazarsfeld.
\newblock {\em Positivity in algebraic geometry. {II}. Positivity for vector
  bundles, and multiplier ideals.}, volume~49 of {\em Ergebnisse der Mathematik
  und ihrer Grenzgebiete. 3. Folge. A Series of Modern Surveys in Mathematics}.
\newblock Springer-Verlag, Berlin, 2004.

\bibitem[Ohn22]{MR4453350}
Masahiro Ohno.
\newblock Nef vector bundles on a projective space or a hyperquadric with the
  first {C}hern class small.
\newblock {\em Rend.\ Circ.\ Mat.\ Palermo Series 2}, 71(2):755--781, 2022.

\bibitem[Ohn23a]{HyperquadricOfDim4c1=2}
Masahiro Ohno.
\newblock Nef vector bundles on a hyperquadric with first {C}hern class two.
\newblock {\em Preprint}, 2023.

\bibitem[Ohn23b]{QuadricSurfacec1=2}
Masahiro Ohno.
\newblock Nef vector bundles on a quadric surface with first {C}hern class
  $(2,2)$.
\newblock {\em Preprint}, 2023.

\bibitem[OT14]{MR3275418}
Masahiro Ohno and Hiroyuki Terakawa.
\newblock A spectral sequence and nef vector bundles of the first {C}hern class
  two on hyperquadrics.
\newblock {\em Ann.\ Univ.\ Ferrara Sez. VII Sci.\ Mat.}, 60(2):397--406, 2014.

\bibitem[Ott88]{ot}
Giorgio Ottaviani.
\newblock Spinor bundles on quadrics.
\newblock {\em Trans. Amer. Math. Soc.}, 307(1):301--316, 1988.

\bibitem[PSW92]{pswnef}
Thomas Peternell, Micha{\l} Szurek, and Jaros{\l}aw~A. Wi{\'{s}}niewski.
\newblock Numerically effective vector bundles with small {C}hern classes.
\newblock In K.~Hulek, T.~Peternell, M.~Schneider, and F.-O. Schreyer, editors,
  {\em Complex Algebraic varieties, Proceedinigs, Bayreuth, 1990}, number 1507
  in Lecture Notes in Math., pages 145--156, Berlin, 1992. Springer.

\bibitem[Wi{\'{s}}89]{w3}
Jaros{\l}aw~A. Wi{\'{s}}niewski.
\newblock Length of extremal rays and generalized adjunction.
\newblock {\em Math. Z.}, 200(3):409--427, 1989.

\end{thebibliography}

\newcommand{\noop}[1]{} \newcommand{\noopsort}[1]{}
  \newcommand{\printfirst}[2]{#1} \newcommand{\singleletter}[1]{#1}
  \newcommand{\switchargs}[2]{#2#1}

\end{document}